\documentclass[a4paper]{amsart}
\usepackage[latin1]{inputenc}
\usepackage{amsmath}
\usepackage{amsthm}
\usepackage{amsfonts}
\usepackage{amssymb}
\usepackage{float}
\usepackage[all]{xy}
\usepackage{tikz}
\usetikzlibrary{arrows,patterns}
\usetikzlibrary{decorations.markings}
\usepackage{lscape}
\usepackage{bbm}
\usepackage{color}
\usepackage{hyperref}
\usepackage[hmargin=2.5cm]{geometry}
\usepackage{mathabx}

\theoremstyle{plain}
\newtheorem{theorem}{Theorem}[section]
\newtheorem{prop}[theorem]{Proposition}
\newtheorem{lem}[theorem]{Lemma}
\newtheorem{corol}[theorem]{Corollary}
\newtheorem{conj}[theorem]{Conjecture}

\theoremstyle{definition}
\newtheorem{defi}[theorem]{Definition}

\newtheorem{rmq}[theorem]{Remark}
\newtheorem{exmp}[theorem]{Example}

\def\pdim{{\rm{pdim}\,}}

\def\irr{{\rm{Irr}\,}}

\def\c{\mathbf{c}}
\def\i{\mathbf{i}}

\def\x{\mathbf{x}}

\def\k{\mathbf{k}}
\def\kQ{\k Q}

\def\fl{{\longrightarrow}\,}
\def\lf{{\longleftarrow}\,}

\def\green{\mathrm{green}}

\def\AA{{\mathcal{A}}}
\def\CC{{\mathcal{C}}}
\def\DD{{\mathcal{D}}}
\def\DDfd{\DD_{\mathrm{fd}}}
\def\HH{{\mathcal{H}}}
\def\II{{\mathcal{I}}}
\def\JJ{{\mathcal{J}}}
\def\KK{{\mathcal{K}}}
\def\PP{{\mathcal{P}}}

\def\TT{{\mathcal{T}}}

\def\h#1{\widehat{#1}}
\def\oo#1{\overrightarrow{#1}}
\def\v#1{\widecheck{#1}}

\def\E{\mathbf {EG}}

\def\Z{{\mathbb{Z}}}

\def\add{{\rm{add}}\,}
\def\per{{\mathrm{per}}\,}

\def\ens#1{\left\{ #1 \right\}}
\def\Ext{{\rm{Ext}}}
\def\End{{\rm{End}}}
\def\Hom{{\rm{Hom}}}

\def\cone{{\rm{Cone}}\,}

\def\Mut{{\mathrm{Mut}}}

\def\modd{\mathrm{mod}\,}

\def\op{{\rm{op}\,}}
\def\min{{\rm{min}\,}}

\def\thick{{\rm{thick}\,}}

\title[On Maximal Green Sequences]{On Maximal Green Sequences} 
\author{T.~Br\"ustle, G.~Dupont and M.~P\'erotin}
\address{Université de Sherbrooke, D\'epartement de Math\'ematiques, 2500 Boul. de l'Universit\'e, Sherbrooke QC J1K 2R1, Canada.}
\email{tbruestl@ubishops.ca}
\address{Bishop's University, Department of Mathematics, 2600 College Street, Sherbrooke QC J1M 1Z7, Canada.}
\email{Thomas.Brustle@usherbrooke.ca}
\address{Universit\'e Denis Diderot Paris 7, Paris, France.}
\email{dupontg@math.jussieu.fr}
\address{Bull SAS, Bruy\`eres-le-Ch\^atel, France.}
\email{matthieu.perotin@bull.net}
\date{\today}

\begin{document}

\begin{abstract}
	Maximal green sequences are particular sequences of quiver mutations appearing in the context of quantum dilogarithm identities and supersymmetric gauge theory. 
	
	Interpreting maximal green sequences as paths in various natural posets arising in representation theory, we prove the finiteness of the number of maximal green sequences for cluster finite quivers, affine quivers and acyclic quivers with at most three vertices. We also give results concerning the possible numbers and lengths of these maximal green sequences.
\end{abstract}

\maketitle

\section{Introduction}
	A maximal green sequence is a certain sequence of quiver mutations given by a sequence $\i = (i_1, \ldots, i_l)$ of vertices in a quiver $Q$. The term ``\emph{maximal green sequence}'' has been coined by Keller in \cite{Keller:dilogarithm} where these sequences are used to obtain quantum dilogarithm identities and refined Donaldson-Thomas invariants. In fact, each sequence $\i$ defines an element ${\mathbb E}(\i)$ in the formal quantum affine space $\mathbb A_Q$ of $q$-commuting variables, and it has been shown in \cite{Keller:dilogarithm} that ${\mathbb E}(\i)={\mathbb E}(\i')$ whenever $\i,\i'$ are maximal green sequences of a Dynkin quiver. For the quiver $Q$ with two vertices and one arrow between them (i.e. $Q$ is of type $A_2$), there are exactly two maximal green sequences $\i$ and $\i'$ of respective lengths two and three, and in this case the equality ${\mathbb E}(\i)={\mathbb E}(\i')$ is a quantum analog of the \emph{pentagon identity} for the Rogers dilogarithm.
	Even in the non-Dynkin case, the invariants ${\mathbb E}(\i)$ can be interpreted as refined DT-invariants following work of Kontsevich-Soibelman; we refer to \cite{Keller:dilogarithm} for a thorough explanation of these ideas.

	Independently, the same sequences of quiver mutations are studied in theoretical physics where they yield the complete spectrum of BPS states, see \cite{ACCERV:BPS,CCV:Braids}.

	Maximal green sequences can also be interpreted as maximal chains in a partially ordered set that arises from a cluster exchange graph once an initial seed is fixed.
	This partial order relation has earlier been studied by Happel and Unger on a subgraph of the cluster exchange graph \cite{Unger:3vertices,Unger:simplicial,HU:tiltinghereditary}, and recently a number of representation-theoretic interpretations of the poset structure of the whole cluster exchange graph has been given \cite{Ladkani:universal,KQ:exchangeCY,KY:simpleminded,IR:tautilting}.

	The theory of cluster algebras is related to numerous other fields, and thus the cluster exchange graph can be interpreted in many ways. For instance one can view it as a generalised associahedron, which is known to carry a poset structure (the \emph{Tamari poset}). However, we will focus in this paper mainly on the combinatorial description of the poset structure by quiver mutations as given in \cite{Keller:dilogarithm}, and we intend to initiate a systematic study of maximal green sequences applying representation-theoretic techniques.

	\subsection*{Organisation of the article}
		In Section \ref{section:greenseq}, we introduce the notion of maximal green sequences in elementary terms and present some general results. When the proofs do not require any further background we present them in this section. When they do require some additional background, they are postponed to Section \ref{section:proofsgreenseq}.

		The short Section \ref{section:BPS} makes the appearance of maximal green sequences explicit in the context of theoretical physics.

		In Section \ref{section:finite}, we study maximal green sequences for quivers of finite cluster type. As before, the proofs requiring additional background are postponed to Section \ref{section:proofsfinite}. 

		Section \ref{section:infinite} presents an analysis of the maximal green sequences for acyclic quivers of infinite representation types; the corresponding proofs are found in Section \ref{section:proofsinfinite}.

		The representation-theoretical background underlying the proofs and (part of) the motivations of this article can be found in Section \ref{section:reptheory} where we recall the various connections between maximal green sequences and some classical posets in representation theory.

		In this spirit, we present in Section \ref{section:HU} additional results on the connections between maximal green sequences and the classical Happel-Unger poset of tilting modules over an algebra, see \cite{HU:tiltinghereditary}. 

		Sections \ref{section:proofsgreenseq}--\ref{section:proofsinfinite} contain the missing proofs.

		Finally, in Appendix \ref{section:examples} we present certain explicit examples which were computed with the \verb\Quiver\ \verb\Mutation Explorer\, see \cite{DupontPerotin:QME}.

\section{Green sequences}\label{section:greenseq}
	Without further specification, \emph{quivers} will always be finite connected oriented graphs and \emph{cluster quivers} will be quivers without loops or oriented 2-cycles. A quiver is called \emph{acyclic} if it has no oriented cycles. Given a quiver $Q$, we denote by $Q_0$ its set of vertices and by $Q_1$ its set of arrows.

	\subsection{Cluster algebras}
		Introduced in \cite{cluster1}, cluster algebras are commutative rings equipped with a distinguished set of generators, the \emph{cluster variables}, gathered into possibly overlapping subsets of pairwise compatible variables, the \emph{clusters}, defined recursively by a combinatorial process, the \emph{mutation}. The dynamics of this mutation process are encoded in a combinatorial datum, the \emph{exchange matrix}. 

		An \emph{exchange matrix} is a matrix $B = (b_{ij}) \in M_{n,n+m}(\Z)$ for some $m,n \geq 0$ such that the \emph{principal part} of $B$, that is, the square submatrix $B^0 = (b_{ij})_{1 \leq i,j \leq n} \in M_n(\Z)$ is \emph{skew-symmetrisable}, that is, there exists a diagonal matrix $D \in M_n(\Z)$ with positive diagonal entries such that $DB^0$ is skew-symmetric. Abusing terminology we say that $B$ itself is \emph{skew-symmetrisable}, or that it is \emph{skew-symmetric} when $B^0$ is so.

		Given a skew-symmetrisable exchange matrix $B \in M_{n,n+m}(\Z)$, we denote by $\AA_B$ the corresponding cluster algebra, see \cite{cluster4} for details.

		\begin{defi}[Matrix mutation]\label{defi:matrixmutation}
			Let $B \in M_{n,n+m}(\Z)$ be skew-symmetrisable. Then for any $1 \leq k \leq n$, the \emph{mutation of $B$ in the direction $k$} is the skew-symmetrisable matrix $\mu_k(B) = (b'_{ij}) \in M_{n,n+m}(\Z)$ given by
			$$
				b'_{ij} = \left\{\begin{array}{ll}
						-b_{ij} & \text{ if } i=k \text{ or } j=k, \\
						\displaystyle b_{ij} + [b_{ik}]_+[b_{kj}]_+-[b_{ik}]_-[b_{kj}]_- & \text{ otherwise,}
				\end{array}\right.
			$$
			where $[x]_+ = \max(x,0)$ and $[x]_- = \min(x,0)$ for any $x \in \Z$.
		\end{defi}

		It is easy to see that $\mu_k(\mu_k(B)) = B$ for any $1 \leq k \leq n$ and that $\mu_k(B)$ is skew-symmetric if and only if $B$ is skew-symmetric. In this latter case, we say that $\mathcal A_B$ is \emph{simply-laced} and it is usually more convenient to use the formalism of \emph{ice quivers} instead of exchange matrices.

	\subsection{Ice quivers and their mutations}
		An \emph{ice quiver} is a pair $(Q,F)$ where $Q$ is a cluster quiver and $F \subset Q_0$ is a (possibly empty) subset of vertices called the \emph{frozen vertices} such that there are no arrows between them. For simplicity, we always assume that $Q_0 = \ens{1, \ldots, n+m}$ and that $F=\ens{n+1, \ldots, n+m}$ for some integers $m,n \geq 0$. If $F$ is empty, we simply write $Q$ for $(Q,\emptyset)$.

		We associate to $(Q,F)$ its \emph{adjacency matrix} $B(Q,F) = (b_{ij}) \in M_{n,n+m}(\Z)$ such that
		$$b_{ij} = |\ens{i \fl j \in Q_1}| - |\ens{j \fl i \in Q_1}|$$
		for any $1 \leq i \leq n$ and any $1 \leq j \leq n+m$.

		The map $(Q,F) \mapsto B(Q,F)$ induces a bijection from the set of ice quivers to the set of skew-symmetric exchange matrices. Therefore, to any ice quiver $(Q,F)$ we can associate the cluster algebra $\mathcal A_{(Q,F)} = \mathcal A_{B(Q,F)}$. 

		\begin{defi}[Quiver mutation]
			Let $(Q,F)$ be an ice quiver and $k \in Q_0$ be a non-frozen vertex. The \emph{mutation of $Q$ at $k$} is defined as the ice quiver $(\mu_k(Q),F)$ where $\mu_k(Q)$ is obtained from $Q$ by applying the following modifications:
			\begin{enumerate}
				\item For any pair of arrows $i \xrightarrow{a} k \xrightarrow{b} j$ in $Q$, add an arrow $i \xrightarrow{[ab]} j$ in $\mu_k(Q)$;
				\item Any arrow $i \xrightarrow{a} k$ in $Q$ is replaced by an arrow $i \xleftarrow{a^*} k$ in $\mu_k(Q)$;
				\item Any arrow $k \xrightarrow{b} j$ in $Q$ is replaced by an arrow $k \xleftarrow{b^*} j$ in $\mu_k(Q)$;
				\item A maximal collection of 2-cycles is removed as well as all the arrows between frozen vertices.
			\end{enumerate}
		\end{defi}

		Then it is easy to see that for any non-frozen vertex $k \in Q_0$, the ice quiver $\mu_k(Q,F)$ is the ice quiver corresponding to the skew-symmetric matrix $\mu_k(B(Q,F))$.

		\begin{exmp}
			Figure \ref{fig:exmpmut} shows an example of successive quiver mutations.
			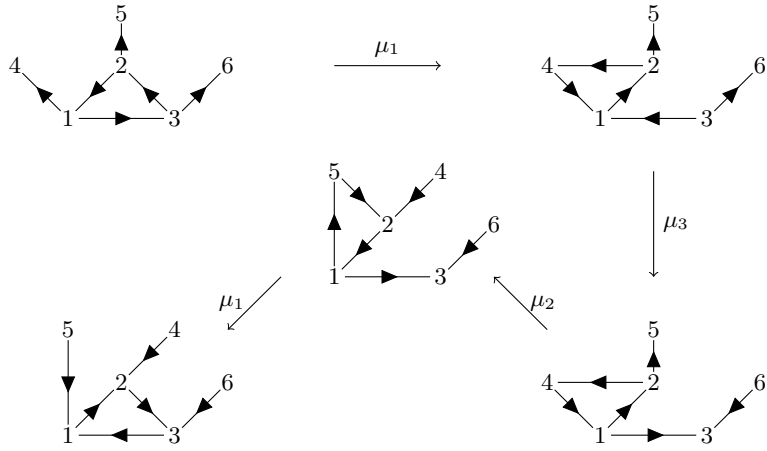
\begin{figure}[htb]
				\begin{center}
					\begin{tikzpicture}[scale = .7]
						\tikzstyle{every node} = [font = \small]
						\begin{scope}[decoration={
							markings,
							mark=at position 0.6 with {\arrow{triangle 45}}}
							] 
							\foreach \x in {0}
							{
								\foreach \y in {0}
								{
									
									\coordinate (1) at (-1+\x,0+\y);
									\coordinate (2) at (0+\x,1+\y);
									\coordinate (3) at (1+\x,0+\y);
									\coordinate (4) at (-2+\x,1+\y);
									\coordinate (5) at (0+\x,2+\y);
									\coordinate (6) at (2+\x,1+\y);

									\draw[postaction={decorate}] (2) -- (1);
									\draw[postaction={decorate}] (3) -- (2);
									\draw[postaction={decorate}] (1) -- (3);

									\draw[postaction={decorate}] (1) -- (4);
									\draw[postaction={decorate}] (2) -- (5);
									\draw[postaction={decorate}] (3) -- (6);

									\fill[white] (1) circle (.2);
									\fill (1) node {1};
									\fill[white] (2) circle (.2);
									\fill (2) node {2};
									\fill[white] (3) circle (.2);
									\fill (3) node {3};
									\fill[white] (4) circle (.2);
									\fill (4) node {4};
									\fill[white] (5) circle (.2);
									\fill (5) node {5};
									\fill[white] (6) circle (.2);
									\fill (6) node {6};

								}
							}

							\draw[->] (4,1) -- (6,1);
							\fill (5,1) node [above] {$\mu_1$};

							\foreach \x in {10}
							{
								\foreach \y in {0}
								{
									
									\coordinate (1) at (-1+\x,0+\y);
									\coordinate (2) at (0+\x,1+\y);
									\coordinate (3) at (1+\x,0+\y);
									\coordinate (4) at (-2+\x,1+\y);
									\coordinate (5) at (0+\x,2+\y);
									\coordinate (6) at (2+\x,1+\y);

									\draw[postaction={decorate}] (1) -- (2);
									\draw[postaction={decorate}] (3) -- (1);

									\draw[postaction={decorate}] (4) -- (1);
									\draw[postaction={decorate}] (2) -- (4);
									\draw[postaction={decorate}] (2) -- (5);
									\draw[postaction={decorate}] (3) -- (6);

									\fill[white] (1) circle (.2);
									\fill (1) node {1};
									\fill[white] (2) circle (.2);
									\fill (2) node {2};
									\fill[white] (3) circle (.2);
									\fill (3) node {3};
									\fill[white] (4) circle (.2);
									\fill (4) node {4};
									\fill[white] (5) circle (.2);
									\fill (5) node {5};
									\fill[white] (6) circle (.2);
									\fill (6) node {6};

								}
							}

							\draw[->] (10,-1) -- (10,-3);
							\fill (10,-2) node [right] {$\mu_3$};

							\foreach \x in {10}
							{
								\foreach \y in {-6}
								{
									
									\coordinate (1) at (-1+\x,0+\y);
									\coordinate (2) at (0+\x,1+\y);
									\coordinate (3) at (1+\x,0+\y);
									\coordinate (4) at (-2+\x,1+\y);
									\coordinate (5) at (0+\x,2+\y);
									\coordinate (6) at (2+\x,1+\y);

									\draw[postaction={decorate}] (1) -- (2);
									\draw[postaction={decorate}] (1) -- (3);

									\draw[postaction={decorate}] (4) -- (1);
									\draw[postaction={decorate}] (2) -- (4);
									\draw[postaction={decorate}] (2) -- (5);
									\draw[postaction={decorate}] (6) -- (3);

									\fill[white] (1) circle (.2);
									\fill (1) node {1};
									\fill[white] (2) circle (.2);
									\fill (2) node {2};
									\fill[white] (3) circle (.2);
									\fill (3) node {3};
									\fill[white] (4) circle (.2);
									\fill (4) node {4};
									\fill[white] (5) circle (.2);
									\fill (5) node {5};
									\fill[white] (6) circle (.2);
									\fill (6) node {6};

								}
							}

							\draw[->] (8,-4) -- (7,-3);
							\fill (7.5,-3.5) node [right] {$\mu_2$};

							\foreach \x in {5}
							{
								\foreach \y in {-3}
								{
									
									\coordinate (1) at (-1+\x,0+\y);
									\coordinate (2) at (0+\x,1+\y);
									\coordinate (3) at (1+\x,0+\y);
									\coordinate (4) at (1+\x,2+\y);
									\coordinate (5) at (-1+\x,2+\y);
									\coordinate (6) at (2+\x,1+\y);

									\draw[postaction={decorate}] (2) -- (1);
									\draw[postaction={decorate}] (1) -- (3);

									\draw[postaction={decorate}] (1) -- (5);
									\draw[postaction={decorate}] (5) -- (2);
									\draw[postaction={decorate}] (4) -- (2);
									\draw[postaction={decorate}] (6) -- (3);

									\fill[white] (1) circle (.2);
									\fill (1) node {1};
									\fill[white] (2) circle (.2);
									\fill (2) node {2};
									\fill[white] (3) circle (.2);
									\fill (3) node {3};
									\fill[white] (4) circle (.2);
									\fill (4) node {4};
									\fill[white] (5) circle (.2);
									\fill (5) node {5};
									\fill[white] (6) circle (.2);
									\fill (6) node {6};

								}
							}

							\draw[->] (3,-3) -- (2,-4);
							\fill (2.5,-3.5) node [left] {$\mu_1$};

							\foreach \x in {0}
							{
								\foreach \y in {-6}
								{
									
									\coordinate (1) at (-1+\x,0+\y);
									\coordinate (2) at (0+\x,1+\y);
									\coordinate (3) at (1+\x,0+\y);
									\coordinate (4) at (1+\x,2+\y);
									\coordinate (5) at (-1+\x,2+\y);
									\coordinate (6) at (2+\x,1+\y);

									\draw[postaction={decorate}] (1) -- (2);
									\draw[postaction={decorate}] (2) -- (3);
									\draw[postaction={decorate}] (3) -- (1);

									\draw[postaction={decorate}] (5) -- (1);
									\draw[postaction={decorate}] (4) -- (2);
									\draw[postaction={decorate}] (6) -- (3);

									\fill[white] (1) circle (.2);
									\fill (1) node {1};
									\fill[white] (2) circle (.2);
									\fill (2) node {2};
									\fill[white] (3) circle (.2);
									\fill (3) node {3};
									\fill[white] (4) circle (.2);
									\fill (4) node {4};
									\fill[white] (5) circle (.2);
									\fill (5) node {5};
									\fill[white] (6) circle (.2);
									\fill (6) node {6};

								}
							}
						\end{scope}
					\end{tikzpicture}
				\end{center}
				\caption{An example of quiver mutations.}\label{fig:exmpmut}
			\end{figure}
		\end{exmp}

		Two ice quivers are called \emph{mutation-equivalent} if one can be obtained from the other by applying a finite number of successive mutations at non-frozen vertices. Since mutations are involutive, this defines an equivalence relation on the set of ice quivers. The equivalence class of an ice quiver $(Q,F)$ is called its \emph{mutation class} and is denoted by $\Mut(Q,F)$.

		Two ice quivers $(Q,F)$ and $(Q',F)$ sharing the same set of frozen vertices are called \emph{isomorphic as ice quivers} if there is an isomorphism of quivers $\phi: Q \fl Q'$ fixing $F$. In this case, we write $(Q,F) \simeq (Q',F)$ and we denote by $[(Q,F)]$ the isomorphism class of the ice quiver $(Q,F)$.

	\subsection{Green sequences}
		From now on, $Q$ will always denote a cluster quiver and we fix a copy $Q_0' = \ens{i' \ | \ i \in Q_0}$ of the set $Q_0$ of vertices in $Q$. We will identify $Q_0$ with the set of integers $\ens{1, \ldots, n}$ and $Q_0'$ with $\ens{n+1, \ldots, 2n}$ in such a way that for any $1 \leq i \leq n$, we have $i' = n+i$.

		\begin{defi}[Framed and coframed quivers]
			The \emph{framed quiver} associated with $Q$ is the quiver $\h Q$ such that:
			\begin{align*}
				\h Q_0 & = Q_0 \sqcup \ens{i' \ | \ i \in Q_0},\\
				\h Q_1 & = Q_1 \sqcup \ens{i \fl i' \ | \ i \in Q_0}.
			\end{align*}

			The \emph{coframed quiver} associated with $Q$ is the quiver $\v Q$ such that:
			\begin{align*}
				\v Q_0 & = Q_0 \sqcup \ens{i' \ | \ i \in Q_0},\\
				\v Q_1 & = Q_1 \sqcup \ens{i' \fl i \ | \ i \in Q_0}.
			\end{align*}
		\end{defi}

		If $Q$ is an arbitrary cluster quiver, both $\h Q$ and $\v Q$ are naturally ice quivers with frozen vertices $Q_0'$. Therefore, by $\Mut(\h Q)$ we always mean the mutation class of the ice quiver $(\h Q, Q_0')$. 

		\begin{defi}[Green and red vertices]
			Let $R \in \Mut(\h Q)$. A non-frozen vertex $i \in R_0$ is called \emph{green} if 
			$$\ens{j' \in Q_0' \ | \ \exists\, j' \fl i \in R_1} = \emptyset.$$
			It is called \emph{red}
			if 
			$$\ens{j' \in Q_0' \ | \ \exists\, i \fl j' \in R_1} = \emptyset.$$
		\end{defi}

		If $R$ is an ice quiver in $\Mut(\h Q)$ with adjacency matrix $B= (b_{ij}) \in M_{n,2n}(\Z)$, the submatrix $\c(R) = (b_{i,n+j})_{1 \leq i,j \leq n}$ is called the \emph{$\c$-matrix} of $R$. For any non-frozen vertex $i \in Q_0$, its $i$th row $\c_i(R)$ is called the \emph{$i$th $\c$-vector} of $R$ and it encodes the number of arrows between $i$ and the frozen vertices in $R$. For instance, we have $\c(\h Q) = I_n$ and $\c(\v Q) = -I_n$. For more details on $\c$-vectors, we refer the reader to \cite{cluster4} where they were introduced and to \cite{NZ:tropicalduality,Najera:cvectors,ST:acyclic,Nagao:DTcluster,Keller:derivedcluster} where they were studied.

		With this terminology, for a quiver $R \in \Mut(\h Q)$, a vertex $i \in Q_0$ is green if and only if the $i$th $\c$-vector $\c_i(R)$ has only non-negative entries and it is red if and only if $\c_i(R)$ has only non-positive entries.

		Given a quiver $R \in \Mut(\h Q)$, we denote by $g(R)$ the number of green vertices in $R$. Note that this number only depends on $[R]$ so that we set $g([R])=g(R)$.

		\begin{theorem}\label{theorem:greenorred}
			Let $Q$ be a cluster quiver and $R \in \Mut(\h Q)$. Then any non-frozen vertex in $R_0$ is either green or red.
		\end{theorem}
		\begin{proof}
			Let $R \in \Mut(\h Q)$. We need to prove that each row of the $\c$-matrix of $R$ is non-zero and that its entries are either all non-negative or all non-positive. This result, known as the \emph{sign-coherence for $\c$-vectors}, was established in the case of skew-symmetric exchange matrices in \cite{DWZ:potentials2}.
		\end{proof}
		
		For skew-symmetrisable exchange matrices the sign-coherence for $\c$-vectors is still conjectural, and so is the non-simply-laced analogue of Theorem \ref{theorem:greenorred}. For partial results concerning this, one might refer to \cite{Demonet:species}.

		\begin{exmp}
			In $\h Q$, every non frozen vertex is green. In $\v Q$, any non-frozen vertex is red.
		\end{exmp}

		\begin{defi}[Green sequences, \cite{Keller:dilogarithm}]
			A \emph{green sequence for $Q$} is a sequence $\i = (i_1, \ldots, i_l) \subset Q_0$ such that $i_1$ is green in $\h Q$ and for any $2 \leq k \leq l$, the vertex $i_k$ is green in $\mu_{i_{k-1}} \circ \cdots \circ \mu_{i_1}(\h Q)$. The integer $l$ is called the \emph{length} of the sequence $\i$ and is denoted by $\ell(\i)$.

			A green sequence $\i = (i_1, \ldots, i_l)$ is called \emph{maximal} if every non-frozen vertex in $\mu_\i(\h Q)$ is red, where $\mu_\i(\h Q) = \mu_{i_l} \circ \cdots \circ \mu_{i_1}(\h Q)$.

			We denote by
			$$\green(Q) = \ens{\i=(i_1, \ldots, i_l) \subset Q_0 \ | \ \i \text{ is a maximal green sequence for }Q}$$
			the set of all maximal green sequences for $Q$.
		\end{defi}
	
		\begin{exmp}
			Figure \ref{fig:exmpgreen} shows that the sequence of mutations considered in Figure \ref{fig:exmpmut} is a maximal green sequence for the oriented triangle. Frozen vertices are coloured in white, green vertices in green and red vertices in red.

			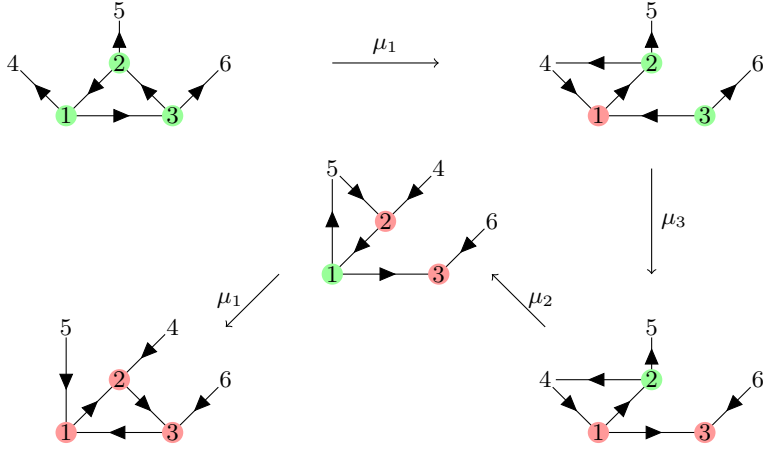
\begin{figure}[htb]
				\begin{center}
					\begin{tikzpicture}[scale = .7]
						\tikzstyle{every node} = [font = \small]
						\begin{scope}[decoration={
							markings,
							mark=at position 0.6 with {\arrow{triangle 45}}}
							] 
							\foreach \x in {0}
							{
								\foreach \y in {0}
								{
									
									\coordinate (1) at (-1+\x,0+\y);
									\coordinate (2) at (0+\x,1+\y);
									\coordinate (3) at (1+\x,0+\y);
									\coordinate (4) at (-2+\x,1+\y);
									\coordinate (5) at (0+\x,2+\y);
									\coordinate (6) at (2+\x,1+\y);

									\draw[postaction={decorate}] (2) -- (1);
									\draw[postaction={decorate}] (3) -- (2);
									\draw[postaction={decorate}] (1) -- (3);

									\draw[postaction={decorate}] (1) -- (4);
									\draw[postaction={decorate}] (2) -- (5);
									\draw[postaction={decorate}] (3) -- (6);

									\fill[green!40] (1) circle (.2);
									\fill (1) node {1};
									\fill[green!40] (2) circle (.2);
									\fill (2) node {2};
									\fill[green!40] (3) circle (.2);
									\fill (3) node {3};
									\fill[white] (4) circle (.2);
									\fill (4) node {4};
									\fill[white] (5) circle (.2);
									\fill (5) node {5};
									\fill[white] (6) circle (.2);
									\fill (6) node {6};

								}
							}

							\draw[->] (4,1) -- (6,1);
							\fill (5,1) node [above] {$\mu_1$};

							\foreach \x in {10}
							{
								\foreach \y in {0}
								{
									
									\coordinate (1) at (-1+\x,0+\y);
									\coordinate (2) at (0+\x,1+\y);
									\coordinate (3) at (1+\x,0+\y);
									\coordinate (4) at (-2+\x,1+\y);
									\coordinate (5) at (0+\x,2+\y);
									\coordinate (6) at (2+\x,1+\y);

									\draw[postaction={decorate}] (1) -- (2);
									\draw[postaction={decorate}] (3) -- (1);

									\draw[postaction={decorate}] (4) -- (1);
									\draw[postaction={decorate}] (2) -- (4);
									\draw[postaction={decorate}] (2) -- (5);
									\draw[postaction={decorate}] (3) -- (6);

									\fill[red!40] (1) circle (.2);
									\fill (1) node {1};
									\fill[green!40] (2) circle (.2);
									\fill (2) node {2};
									\fill[green!40] (3) circle (.2);
									\fill (3) node {3};
									\fill[white] (4) circle (.2);
									\fill (4) node {4};
									\fill[white] (5) circle (.2);
									\fill (5) node {5};
									\fill[white] (6) circle (.2);
									\fill (6) node {6};

								}
							}

							\draw[->] (10,-1) -- (10,-3);
							\fill (10,-2) node [right] {$\mu_3$};

							\foreach \x in {10}
							{
								\foreach \y in {-6}
								{
									
									\coordinate (1) at (-1+\x,0+\y);
									\coordinate (2) at (0+\x,1+\y);
									\coordinate (3) at (1+\x,0+\y);
									\coordinate (4) at (-2+\x,1+\y);
									\coordinate (5) at (0+\x,2+\y);
									\coordinate (6) at (2+\x,1+\y);

									\draw[postaction={decorate}] (1) -- (2);
									\draw[postaction={decorate}] (1) -- (3);

									\draw[postaction={decorate}] (4) -- (1);
									\draw[postaction={decorate}] (2) -- (4);
									\draw[postaction={decorate}] (2) -- (5);
									\draw[postaction={decorate}] (6) -- (3);

									\fill[red!40] (1) circle (.2);
									\fill (1) node {1};
									\fill[green!40] (2) circle (.2);
									\fill (2) node {2};
									\fill[red!40] (3) circle (.2);
									\fill (3) node {3};
									\fill[white] (4) circle (.2);
									\fill (4) node {4};
									\fill[white] (5) circle (.2);
									\fill (5) node {5};
									\fill[white] (6) circle (.2);
									\fill (6) node {6};

								}
							}

							\draw[->] (8,-4) -- (7,-3);
							\fill (7.5,-3.5) node [right] {$\mu_2$};

							\foreach \x in {5}
							{
								\foreach \y in {-3}
								{
									
									\coordinate (1) at (-1+\x,0+\y);
									\coordinate (2) at (0+\x,1+\y);
									\coordinate (3) at (1+\x,0+\y);
									\coordinate (4) at (1+\x,2+\y);
									\coordinate (5) at (-1+\x,2+\y);
									\coordinate (6) at (2+\x,1+\y);

									\draw[postaction={decorate}] (2) -- (1);
									\draw[postaction={decorate}] (1) -- (3);

									\draw[postaction={decorate}] (1) -- (5);
									\draw[postaction={decorate}] (5) -- (2);
									\draw[postaction={decorate}] (4) -- (2);
									\draw[postaction={decorate}] (6) -- (3);

									\fill[green!40] (1) circle (.2);
									\fill (1) node {1};
									\fill[red!40] (2) circle (.2);
									\fill (2) node {2};
									\fill[red!40] (3) circle (.2);
									\fill (3) node {3};
									\fill[white] (4) circle (.2);
									\fill (4) node {4};
									\fill[white] (5) circle (.2);
									\fill (5) node {5};
									\fill[white] (6) circle (.2);
									\fill (6) node {6};

								}
							}

							\draw[->] (3,-3) -- (2,-4);
							\fill (2.5,-3.5) node [left] {$\mu_1$};

							\foreach \x in {0}
							{
								\foreach \y in {-6}
								{
									
									\coordinate (1) at (-1+\x,0+\y);
									\coordinate (2) at (0+\x,1+\y);
									\coordinate (3) at (1+\x,0+\y);
									\coordinate (4) at (1+\x,2+\y);
									\coordinate (5) at (-1+\x,2+\y);
									\coordinate (6) at (2+\x,1+\y);

									\draw[postaction={decorate}] (1) -- (2);
									\draw[postaction={decorate}] (2) -- (3);
									\draw[postaction={decorate}] (3) -- (1);

									\draw[postaction={decorate}] (5) -- (1);
									\draw[postaction={decorate}] (4) -- (2);
									\draw[postaction={decorate}] (6) -- (3);

									\fill[red!40] (1) circle (.2);
									\fill (1) node {1};
									\fill[red!40] (2) circle (.2);
									\fill (2) node {2};
									\fill[red!40] (3) circle (.2);
									\fill (3) node {3};
									\fill[white] (4) circle (.2);
									\fill (4) node {4};
									\fill[white] (5) circle (.2);
									\fill (5) node {5};
									\fill[white] (6) circle (.2);
									\fill (6) node {6};

								}
							}
						\end{scope}
					\end{tikzpicture}
				\end{center}
				\caption{An example of a maximal green sequence.}\label{fig:exmpgreen}
			\end{figure}
		\end{exmp}

		We refer the reader willing to compute more examples to Bernhard Keller's java applet \cite{Keller:javaapplet} or to the \verb\Quiver Mutation Explorer\ \cite{DupontPerotin:QME}.

	\subsection{The oriented exchange graph}
		The following proposition will be proved in Section \ref{section:proofsgreenseq}.
		\begin{prop}\label{prop:uniquered}
			Let $Q$ be a cluster quiver and let $R \in \Mut(\h Q)$.
			\begin{enumerate}
				\item If all the non-frozen vertices in $R_0$ are green, then $R \simeq \h Q$ as ice quivers.
				\item If all the non-frozen vertices in $R_0$ are red, then $R \simeq \v Q$ as ice quivers.
			\end{enumerate}
		\end{prop}

		\begin{defi}[Oriented exchange graph]
			The \emph{oriented exchange graph} of $Q$ is the oriented graph $\oo\E(Q)$ whose vertices are the isomorphism classes $[R]$ of ice quivers $R \in \Mut(\h Q)$ and where there is an arrow $[R] \fl [R']$ in $\oo \E (Q)$ if and only if there exists a green vertex $k \in R_0$ such that $\mu_k(R) \simeq R'$.
		\end{defi}

		In \cite{cluster2}, Fomin and Zelevinsky introduced the (unoriented) \emph{exchange graph} of $Q$ as the dual graph $\E(Q)$ of the cluster complex $\Delta(\AA_Q)$ of the cluster algebra $\AA_Q$ associated with $Q$. Vertices in $\E(Q)$ are labelled by the clusters in $\AA_Q$ and two clusters in $\E(Q)$ are joined by an edge if and only if they differ by a single cluster variable. Then $\oo\E(Q)$ is an orientation of $\E(Q)$ corresponding to the choice of an initial seed in $\AA_Q$ with exchange matrix $B(Q)$. The orientation is defined as follows. Let $\x$ and $\x'$ be two adjacent clusters in $\E(Q)$ corresponding respectively to $[R]$ and $[R']$ in $\oo\E(Q)$. Assume that $\x$ and $\x'$ differ by a single cluster variable $x_i$, so that $R' \simeq \mu_i(R)$. Then the edge joining $\x$ and $\x'$ in $\E(Q)$ is oriented towards $\x'$ if $i$ is green in $R$ and towards $\x$ otherwise.

		As $\E(Q)$ is an $n$-regular graph, if $[R]$ is a vertex in $\oo\E(Q)$, then there are $g([R])$ arrows starting at $[R]$ in $\oo\E(Q)$ and $n-g([R])$ arrows ending at $[R]$ in $\oo\E(Q)$ (which, by Theorem \ref{theorem:greenorred}, correspond to the red vertices in $R$).

		\begin{corol}\label{corol:uniquesource}
			Let $Q$ be a cluster quiver. Then:
			\begin{enumerate}
				\item $\oo\E(Q)$ has a unique source, which is $[\h Q]$.
				\item $\oo\E(Q)$ has a sink if and only if $[\v Q]$ is a vertex in $\oo\E(Q)$ and in this case $[\v Q]$ is the unique sink.
			\end{enumerate}
		\end{corol}
		\begin{proof}
			$[\h Q]$ belongs to $\oo\E(Q)$ by construction and it is a source in $\oo\E(Q)$ since all the vertices in $\h Q$ are green. If $[R]$ is another source, then all the vertices in $R$ are green and therefore it follows from Proposition \ref{prop:uniquered} that $R \simeq \h Q$, proving the first point. Now if $[R]$ is a sink in $\oo\E(Q)$, then all its vertices are red and therefore, it follows from Proposition \ref{prop:uniquered} that $R \simeq \v Q$, proving the second point. Conversely, if $[\v Q]$ is in $\oo\E(Q)$, then it is a sink since all its non-frozen vertices are red.
		\end{proof}

		The following statement rephrases Corollary \ref{corol:uniquesource}:
		\begin{prop}\label{prop:bijmaxgreenpaths}
			Let $Q$ be a cluster quiver. Then $\green(Q) \neq \emptyset$ if and only if there is a sink in $\oo\E(Q)$. In this case, there is a natural bijection between $\green(Q)$ and the set of oriented paths in $\oo\E(Q)$ from its unique source to its unique sink. 
		\end{prop}

		As it is explained in Section \ref{section:reptheory}, $\oo\E(Q)$ is isomorphic to the Hasse graph of various partially ordered sets. In particular, it has the following essential property:
		\begin{prop}\label{prop:acyclicity}
			Let $Q$ be a cluster quiver. Then $\oo\E(Q)$ has no oriented cycles. 
		\end{prop}

	\subsection{Existence, finiteness and lengths}
		Let $Q$ be a cluster quiver. We recall that if $\i = (i_1, \ldots, i_l)$ is a green sequence for $Q$, then the integer $l$ is called the \emph{length} of $\i$ and is denoted by $\ell(\i)$. For any $l \geq 0$, we set
		$$\green_l(Q) = \ens{\i \in \green(Q) \ | \ \ell(\i) = l},$$
		$$\green_{\leq l}(Q) = \ens{\i \in \green(Q) \ | \ \ell(\i) \leq l}$$
		and
		$$\ell_{\min}(Q) = \min \ens{l \geq 0 \ | \ \green_l(Q) \neq \emptyset} \in \Z_{\geq 0},$$
		$$\ell_{\max}(Q) = \max \ens{l \geq 0 \ | \ \green_l(Q) \neq \emptyset} \in \Z_{\geq 0} \sqcup \ens{\infty},$$
		with the conventions that $\ell_{\min}(Q) = \ell_{\max}(Q) = 0$ if $\green(Q)$ is empty.

		It is clear that if $Q$ and $Q'$ are isomorphic quivers, then the isomorphism $\phi:Q \fl Q'$ induces an isomorphism $\oo\E(Q) \fl \oo\E(Q')$ so that $\green_l(Q) = \green_l(Q')$ for any $l \geq 1$. The following proposition shows a similar result for oppositions:
		\begin{prop}\label{prop:opposition}
			Let $Q$ be a cluster quiver. Then for any $l \geq 1$, there exists a natural bijection 
			$$\green_l(Q) \leftrightarrow \green_l(Q^\op).$$
		\end{prop}
		\begin{proof}
			Let $\i = (i_1, \ldots, i_l)$ be a maximal green sequence. Then there exists $\pi \in \mathfrak S_{Q_0}$ such that $\mu_{\i}(\h Q) = \pi \cdot \v Q$. Moreover, since $\pi$ fixes the frozen vertices and since the only arrows between frozen and non-frozen vertices in $\v Q$ are the $i' \fl i$ for $i \in Q_0$, the permutation $\pi$ is uniquely determined. Therefore we have $\mu_{\pi^{-1}(i_1)} \circ \cdots \circ \mu_{\pi^{-1}(i_l)}(\v Q) = \h Q$ where $\pi^{-1}(i_l)$ is red in $\v Q$ and for any $2 \leq k \leq l$, the vertex $\pi^{-1}(i_k)$ is red in $\mu_{\pi^{-1}(i_{k+1})} \circ \cdots \circ \mu_{\pi^{-1}(i_l)}(\v Q)$. Since the mutations commute with taking opposite quivers, $\pi^{-1}(i_l)$ is green in $(\v Q)^{\op}$, the vertex $\pi^{-1}(i_k)$ is green in $\mu_{\pi^{-1}(i_{k+1})} \circ \cdots \circ \mu_{\pi^{-1}(i_l)}((\v Q)^{\op})$ for any $2 \leq k \leq l$ and $\mu_{\pi^{-1}(i_{1})} \circ \cdots \circ \mu_{\pi^{-1}(i_l)}((\v Q)^{\op})$ has only red vertices. Since $(\v Q)^{\op} = \h{Q^\op}$, it follows that $(\pi^{-1}(i_{l}), \ldots, \pi^{-1}(i_1))$ is a maximal green sequence for $Q^{\op}$. We therefore get a map $\green_l(Q) \fl \green_l(Q^\op)$ and applying the same argument to $Q^{\op}$, we get its inverse. Therefore, it is a bijection.
		\end{proof}

		\begin{lem}\label{lem:gRgR'}
			Let $Q$ be a cluster quiver and let $R,R' \in \Mut(\h Q)$ such that $[R] \fl [R']$ in $\oo\E(Q)$. Then $g([R']) \geq g([R])-1$.
		\end{lem}
		\begin{proof}
			Assume that $R' = \mu_k(R)$ for some green vertex $k$ in $R$. In order to prove the statement, it is enough to prove that any green vertex in $R$ which is different from $k$ is also green in $R'$. We let $B=B(R)$ and $B'=B(R')$ be the corresponding adjacency matrices. Let $i$ be a green vertex in $R$ and let $f$ be a frozen vertex in $R$. Since $i$ is green in $R$, we have $b_{if} \geq 0$ and also, since $k$ is green in $R$, we have $b_{kf} \geq 0$. Therefore,
			\begin{align*}
				b'_{if} 
					& = b_{if} + [b_{ik}]_+[b_{kf}]_+ - [b_{ik}]_-[b_{kf}]_- \\
					& = b_{if} + [b_{ik}]_+[b_{kf}]_+ \\
					& \geq b_{if} \geq 0
			\end{align*}
			so that $i$ is green in $R'$.
		\end{proof}

		\begin{rmq}
			Note that under the hypothesis of Lemma \ref{lem:gRgR'}, it may happen that $g([R']) > g([R])-1$ since a red vertex in $R$ can turn green in $R'$, see for instance the penultimate mutation in Figure \ref{fig:exmpgreen}.
		\end{rmq}

		\begin{corol}\label{corol:lowerbound}
			Let $Q$ be a cluster quiver. If $\green(Q) \neq \emptyset$, then $\ell_{\min}(Q) \geq |Q_0|$.
		\end{corol}
		\begin{proof}
			By definition, in a maximal green sequence $\i = (i_1, \ldots, i_l)$, we have $g(\mu_{\i}(\h Q)) = 0$ whereas $g(\h Q) = |Q_0|$. Therefore, it follows from Lemma \ref{lem:gRgR'} that $l \geq |Q_0|$.
		\end{proof}

		\begin{exmp}
			\begin{enumerate}
				\item Let $Q$ be the quiver $1 \fl 2 \fl 3$. Then $\i = (123)$ is a maximal green sequence and therefore $\ell_{\min}(Q) = 3 = |Q_0|$.
				\item Let $Q' = \mu_2(Q)$ be the cyclic quiver with three vertices. Then it is easily verified that $\ell_{\min}(Q') = 4 > 3$ so that Corollary \ref{corol:lowerbound} only provides a lower bound for $\ell_{\min}$.
			\end{enumerate}
			Note also that these examples show that $\ell_{\min}$ is not invariant under mutations. The same will appear to be true for $\ell_{\max}$.
		\end{exmp}

		We recall that for a quiver $Q$, an \emph{admissible numbering of $Q_0$ by sources} (resp. \emph{by sinks}) is an $n$-tuple $(i_1, \ldots, i_{n})$ such that $Q_0 = \ens{i_1, \ldots, i_{n}}$ and where $i_1$ is a source (resp. a sink) in $Q$ and such that for any $2 \leq k \leq {n}$, the vertex $i_k$ is a source (resp. a sink) in $\mu_{i_{k-1}} \circ \cdots \circ \mu_{i_1}(Q)$.

		\begin{lem}\label{lem:lminacyclic}
			Let $Q$ be an acyclic quiver. Then any admissible numbering of $Q_0$ by sources is a maximal green sequence. In particular, $\green(Q) \neq \emptyset$ and $\ell_{\min}(Q) = |Q_0|$.
		\end{lem}
		\begin{proof}
			Since $Q$ is acyclic, it is well-known that there is at least one admissible numbering of $Q_0$ by sources. Let $\i = (i_1, \ldots, i_{n})$ be such a numbering. Without loss of generality, we assume that this admissible numbering is $(1, \ldots, n)$. For any $1 \leq k \leq n$, we let $B^{(k)}$ be the adjacency matrix of $R^{(k)} = \mu_{k} \circ \cdots \circ \mu_1(\h Q)$ and $Q^{(k)} = \mu_{k} \circ \cdots \circ \mu_1(Q)$. We prove by induction on $k$ that the green vertices in $R^{(k)}$ are precisely $\ens{k+1, \ldots, n}$.

			Let $i \neq k$ be non-frozen vertices and $f$ be a frozen vertex. We have 
			$$b^{(k)}_{i,f} = b^{(k-1)}_{i,f} + [b^{(k-1)}_{i,k}]_+[b^{(k-1)}_{k,f}]_+ - [b^{(k-1)}_{i,k}]_-[b^{(k-1)}_{k,f}]_-.$$
			Since $k$ is a source in $Q^{(k-1)}$, it follows that $b^{(k-1)}_{i,k} \leq 0$. Also, by induction hypothesis $k$ is green in $R^{(k-1)}$ so that $b^{(k-1)}_{k,f} \geq 0$. Therefore, $b^{(k)}_{i,f} = b^{(k-1)}_{i,f}$ so that a non-frozen vertex $i \neq k$ is green (or red, respectively) in $R^{(k)}$ if and only if it is green (or red, respectively) in $R^{(k-1)}$. Moreover, $b^{(k)}_{k,k+n} = -b^{(k-1)}_{k,k+n} = -b_{k,k+n} = -1$ so that $k$ is red in $R^{(k)}$ whereas it was green in $R^{(k-1)}$. Thus, the green vertices in $R^{(k)}$ are exactly $\ens{k+1, \ldots, n}$. In particular, $(1, \ldots, n)$ is a maximal green sequence for $Q$.
		\end{proof}

		In general, it is not true that $\green(Q) \neq \emptyset$ for an arbitrary quiver $Q$. For instance, we have the following proposition, which will be proved in Section \ref{section:proofsgreenseq}:
		\begin{prop}\label{prop:doubletriangle}
			The quiver 
			$$
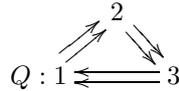
\xymatrix@=1em{
					& 2 \ar@<+2pt>[rd] \ar@<-2pt>[rd] \\
				Q: 1 \ar@<+2pt>[ru] \ar@<-2pt>[ru] && 3 \ar@<+2pt>[ll] \ar@<-2pt>[ll]
			}$$
			has no maximal green sequences.
		\end{prop}

		More generally, a representation-theoretic criterion for the non-existence of maximal green sequences is given in Proposition \ref{prop:nogreeninfinite}. This in particular enables us to show that the McKay quiver 
		$$
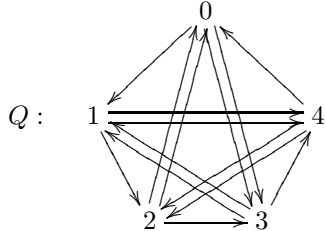
\xymatrix@=1em{
						&&& 0 \ar[lldd] \ar@<+2pt>[rdddd] \ar@<-2pt>[rdddd]\\ \\
			Q:	& 1 \ar[rdd] \ar@<+2pt>[rrrr]\ar@<-2pt>[rrrr] &&&& 4 \ar[lluu] \ar@<+2pt>[ddlll] \ar@<-2pt>[ddlll]\\ \\
					&& 2 \ar[rr] \ar@<+2pt>[ruuuu]\ar@<-2pt>[ruuuu] && 3 \ar[ruu] \ar@<+2pt>[llluu]\ar@<-2pt>[llluu]
		}
		$$
		considered in \cite{VVdB:nondegenerateqps} has no maximal green sequences either, see Example \ref{exmp:VdB}.

		The quiver in Proposition \ref{prop:doubletriangle} is the quiver associated with any triangulation of the once-punctured torus, see \cite{FST:surfaces}. We will see in Section \ref{ssection:sphere} another example of a surface without boundary, namely the sphere with four punctures, for which there do exist maximal green sequences.

	\subsection{A conjecture on lengths}
		Based on the various examples we computed, we conjecture the following.
		\begin{conj}\label{conj:intervals}
			Let $Q$ be a cluster quiver. Then the set
			$$\ens{l \in \Z_{\geq 0} \ | \ \green_l(Q) \neq \emptyset}$$
			is an interval in $\Z$. 
		\end{conj}

		Given a cluster quiver $Q$, the \emph{empirical maximal length} is 
		$$\ell^0_{\max}(Q) = \max \ens{l \geq 1 \ | \ \green_{k}(Q) \neq \emptyset \text{ for any } k \text{ s.t. } \ell_{\min}(Q) \leq k \leq l}$$
		and we let
		$$\green^0(Q) = \green_{\leq \ell^0_{\max}(Q)}(Q),$$
		with the convention that $\ell^0_{\max}(Q) = 0$ if $\green(Q) = \emptyset$. 

		With this terminology, Conjecture \ref{conj:intervals} states that $\ell^0_{\max}(Q) = \ell_{\max}(Q)$ and $\green^0(Q) = \green(Q)$. We will see several examples in which this appears. 
		
		The motivation for introducing the empirical maximal length is that it is easy to determine in practice: let $l$ be the smallest integer such that $\green_{l}(Q) \neq \emptyset$ and $\green_{l+1}(Q) = \emptyset$, then $l=\ell^0_{\max}(Q)$. Therefore, if Conjecture \ref{conj:intervals} holds, it is enough to find such an $l$ to determine $\green(Q)$. This is the strategy we used in the computations whose results are given in Appendix \ref{section:examples}.

		Note that Conjecture \ref{conj:intervals} does not hold true in the non-simply-laced case, as is shown for instance in Appendix \ref{ssection:rank2}. 

\section{Maximal green sequences and BPS quivers}\label{section:BPS}
	As we already mentioned, maximal green sequences appear independently in theoretical physics, implicitly in \cite{GMN:WKB} or more explicitly in \cite{ACCERV:BPS}. In order to make the connection clear, we present in this short section a precise dictionary between the formal definition we gave in the previous section, and the physics review given in \cite[\S 4.2]{CCV:Braids}. 

	We fix a cluster quiver $Q$. Vertices in $Q$ are called \emph{nodes} in \cite{CCV:Braids}. 

	For simplicity, we identify the set $Q_0$ of vertices with $\ens{1, \ldots, n}$. We let $\ens{\gamma_i}_{1 \leq i \leq n}$ denote the canonical basis of $\Z^n$. In the terminology of \cite{CCV:Braids}, for any $R \in \Mut(\h Q)$ and for any $1 \leq i \leq n$, the $i$th $\c$-vector $\c_i(R) \in \Z^n$ is called the \emph{charge at node $i$}. Therefore, the charges in $\h Q$ are $\gamma_1, \ldots, \gamma_n$. 

	For any quiver $R \in \Mut(\h Q)$ and for any $1 \leq k \leq n$, the charge at node $k$ in $R$ is $\c_k(R) = \sum_{i=1}^n c_{k;i}(R) \gamma_i$ where $c_{k;i}(R) \in \Z$ for any $i$. It follows from the sign-coherence for $\c$-vectors (see Theorem \ref{theorem:greenorred}) that either $c_{k;i}(R) \leq 0$ for every $i$, in which case $k$ is green in $R$, or $c_{k;i}(R) \geq 0$ for every $i$, in which case $k$ is red in $R$. Moreover, if $k$ is green in $R$, then the $\c$-vectors of $\mu_k(R)$ are precisely given in terms of those of $R$ by the rule for charges given in \cite[(4.4)]{CCV:Braids}.

	Now, the sequences of mutations considered in \cite{CCV:Braids} for capturing complete BPS spectra are those for which:
	\begin{enumerate}
		\item[(G1)] the initial quiver appears with node charges $\gamma_i$;
		\item[(G2)] the final quiver appears with node charges $-\gamma_i$;
		\item[(G3)] At each step we may mutate on any node whose charge can be expressed as $\gamma = \sum_i c_i \gamma_i$ where $c_i \geq 0$ for any $1 \leq i \leq n$.
	\end{enumerate}

	Therefore, with our terminology, (G1) implies that the initial quiver $R$ has only green vertices, so that $[R] = [\h Q]$ according to Proposition \ref{prop:uniquered}, (G2) implies that the final quiver $R'$ has only red vertices so that $[R] = [\v Q]$ according to Proposition \ref{prop:uniquered}. Finally, (G3) says that at each step in the mutation sequence $\h Q \simeq R \xrightarrow{\mu_{i_1}} R^{(1)} \xrightarrow{\mu_{i_l}} \cdots \xrightarrow{\mu_{i_1}} R^{(l)} \simeq \v Q$, we mutated at a green vertex. Therefore, the sequences considered in \cite{CCV:Braids} are precisely the maximal green sequences of $Q$.

\section{The finite cluster type}\label{section:finite}
	It was proved in \cite{cluster2} that a cluster algebra $\AA_Q$ associated with a cluster quiver $Q$ has finitely many cluster variables if and only if $Q$ is mutation-equivalent to a Dynkin quiver $\oo\Delta$. In this case, $Q$ is called of \emph{finite cluster type} and it is known that the number of cluster variables in $\AA_Q$ equals the number of \emph{almost positive roots} of the Dynkin quiver $\oo\Delta$, where the set $\Phi_{\geq -1}(\oo\Delta)$ of almost positive roots of $\oo\Delta$ is the disjoint union of the set $\Phi_+(\oo\Delta)$ of positive roots with the set of negative simple roots.

	\begin{theorem}\label{theorem:finitetype}
		Let $Q$ be a quiver of finite cluster type. Then
		$$|Q_0| \leq |\green(Q)| < \infty.$$
	\end{theorem}
	\begin{proof}
		Since $Q$ is of finite cluster type, the exchange graph $\E(Q)$ is finite. Moreover, we know from Proposition \ref{prop:acyclicity} that $\oo\E(Q)$ is acyclic. Hence, it contains only finitely many oriented paths and thus it follows from Proposition \ref{prop:bijmaxgreenpaths} that $\green(Q)$ is finite.

		Now since $\oo\E(Q)$ is a finite acyclic oriented graph, it necessarily has at least one sink and one source and by Corollary \ref{corol:uniquesource}, it has a unique sink, corresponding to $[\h Q]$, and a unique source, corresponding to $[\v Q]$. The underlying graph of $\oo\E(Q)$ is $|Q_0|$-regular so that there are exactly $|Q_0|$ distinct arrows starting at $[\h Q]$. Since $\oo\E(Q)$ is finite, each of these arrows gives rise to at least one oriented path from the unique sink to the unique source and therefore we obtain at least $|Q_0|$ distinct oriented paths from the unique source to the unique sink in $\oo\E(Q)$, that is, $|Q_0| \leq |\green(Q)|$.
	\end{proof}

	\begin{rmq}
		\begin{enumerate}
			\item If $Q$ is a cluster quiver such that $|Q_0| = 1$, then clearly $|\green(Q)| = 1$.
			\item If $Q$ is a (connected) cluster quiver such that $|Q_0| = 2$, then it is shown in Lemma \ref{lem:2vertices} that $\green(Q)$ has two elements of respective lengths 2 and 3 in the finite cluster type and a unique element, necessarily of length two, in the other cases.
			\item If $Q$ is a cluster quiver of finite cluster type such that $|Q_0| > 2$, then it $|\green(Q)|>|Q_0|$ in general. In particular, one can observe that for linearly oriented quivers $Q_n$ of type $A_n$, the cardinality $|\green(Q_n)|$ seems to grow exponentially as a function of $n$, see \cite{DupontPerotin:QME}.
		\end{enumerate}
	\end{rmq}

	\begin{rmq}
		If $Q$ is of finite cluster type, then a rough analysis provides an upper bound for $\ell_{\max}(Q)$. Namely, if we set
		$$\chi(Q) = \left|\ens{[R] \ | \ R \in \Mut(\h Q)}\right|,$$
		then we have 
		$$\ell_{\max}(Q) \leq |Q_0| \cdot \left(|Q_0|-1\right)^{\chi(Q)-2}.$$
		Indeed, an oriented path on $\oo\E(Q)$ starts at $[\h Q]$ where we have $|Q_0|$ choices of directions and then, it passes at most once through any vertex in $\oo\E(Q)$ distinct from $[\h Q]$ and $[\v Q]$. There are $\chi(Q)-2$ such vertices and at each such vertex $[R]$, there are at most $|Q_0|-1$ possible directions (since in order to leave $[R]$, we cannot use backwards the arrow we just used in order to arrive at $[R]$).
	\end{rmq}

	In general, these upper and lower bounds are not optimal but in the acyclic case, we can sharpen the result with the following theorem whose proof is given in Section \ref{section:proofsfinite}:
	\begin{theorem}\label{theorem:Dynkinminmax}
		Let $Q$ be a Dynkin quiver. Then:
		\begin{enumerate}
			\item $\ell_{\min}(Q) = |Q_0|$,
			\item $\ell_{\max}(Q) = |\Phi_+(Q)|$,
		\end{enumerate}
		where $\Phi_+(Q)$ is the set of \emph{positive roots} of $Q$.
	\end{theorem}

	\begin{exmp}
		We show below the oriented exchange graphs for some quivers in the mutation class of type $A_3$. The labels on the faces correspond to denominators of the cluster variables in the corresponding clusters expressed in the seed with exchange matrix $B(Q)$. The unique source is circled in green and the unique sink is circled in red.
		\begin{figure}[H]
			\begin{center}
				\begin{tikzpicture}[scale = .75]
					\begin{scope}[decoration={
								markings,
								mark=at position 0.6 with {\arrow{triangle 60}}}
								] 
						\foreach \y in {0}
						{

							\fill (3,\y+1) circle (.1);
							\fill (3,\y-1) circle (.1);

							\fill (7,\y+1) circle (.1);
							\fill (7,\y-1) circle (.1);

							\fill (11,\y+1) circle (.1);
							\fill (11,\y-1) circle (.1);

							\fill (7,\y-3) circle (.1);
							\fill (7,\y+3) circle (.1);

							\foreach \x in {2,4,6,8,10,12}
							{
								\fill (\x,\y) circle (.1);
							}

							\draw[thick,green] (3,\y-1) circle (.2);
							\draw[thick,red] (11,\y+1) circle (.2);
				
							\draw[postaction={decorate}] (2,\y) -- (3,\y+1);
							\draw[postaction={decorate}] (4,\y) -- (3,\y+1);

							\draw[postaction={decorate}] (3,\y-1) -- (2,\y);
							\draw[postaction={decorate}] (3,\y-1) -- (4,\y);

							\draw[postaction={decorate}] (4,\y) -- (6,\y);

							\draw[postaction={decorate}] (6,\y) -- (7,\y+1);
							\draw[postaction={decorate}] (6,\y) -- (7,\y-1);
							\draw[postaction={decorate}] (7,\y+1) -- (8,\y);
							\draw[postaction={decorate}] (7,\y-1) -- (8,\y);

							\draw[postaction={decorate}] (8,\y) -- (10,\y);

							\draw[postaction={decorate}] (10,\y) -- (11,\y+1);
							\draw[postaction={decorate}] (11,\y-1) -- (10,\y);
							\draw[postaction={decorate}] (12,\y) -- (11,\y+1);
							\draw[postaction={decorate}] (11,\y-1) -- (12,\y);

							\draw[postaction={decorate}] (2,\y) -- (1,\y);
							\draw[postaction={decorate}] (13,\y) -- (12,\y);

							\draw[postaction={decorate}] (3,\y+1) -- (7,\y+3);
							\draw[postaction={decorate}] (7,\y+1) -- (7,\y+3);
							\draw[postaction={decorate}] (7,\y+3) -- (11,\y+1);

							\draw[postaction={decorate}] (3,\y-1) -- (7,\y-3);
							\draw[postaction={decorate}] (7,\y-3) -- (7,\y-1);
							\draw[postaction={decorate}] (7,\y-3) -- (11,\y-1);

							\fill (3,\y) node {$-\alpha_2$};
							\fill (4,\y+2.5) node {$\alpha_1$};
							\fill (4,\y-2.5) node {$-\alpha_3$};

							\fill (5,\y+1) node {$\alpha_3$};
							\fill (5,\y-1) node {$-\alpha_1$};

							\fill (7,\y) node {$\alpha_2 + \alpha_3$};

							\fill (9,\y+1) node {$\alpha_1+\alpha_2+\alpha_3$};
							\fill (9,\y-1) node {$\alpha_2$};

							\fill (11,\y) node {$\alpha_1+\alpha_2$};

						}
					\end{scope}
				\end{tikzpicture}
				\caption{The oriented exchange graph of $1 \fl 2 \fl 3$.}
			\end{center}
		\end{figure}

		\begin{figure}[H]
			\begin{center}
				\begin{tikzpicture}[scale = .75]
					\begin{scope}[decoration={
								markings,
								mark=at position 0.6 with {\arrow{triangle 60}}}
								] 
						\foreach \y in {0}
						{

							\fill (3,\y+1) circle (.1);
							\fill (3,\y-1) circle (.1);

							\fill (7,\y+1) circle (.1);
							\fill (7,\y-1) circle (.1);

							\fill (11,\y+1) circle (.1);
							\fill (11,\y-1) circle (.1);

							\fill (7,\y-3) circle (.1);
							\fill (7,\y+3) circle (.1);

							\foreach \x in {2,4,6,8,10,12}
							{
								\fill (\x,\y) circle (.1);
							}

							\draw[thick,green] (7,\y-3) circle (.2);
							\draw[thick,red] (7,\y+3) circle (.2);
				
							\draw[postaction={decorate}] (2,\y) -- (3,\y+1);
							\draw[postaction={decorate}] (4,\y) -- (3,\y+1);

							\draw[postaction={decorate}] (3,\y-1) -- (2,\y);
							\draw[postaction={decorate}] (3,\y-1) -- (4,\y);

							\draw[postaction={decorate}] (6,\y) -- (4,\y);

							\draw[postaction={decorate}] (6,\y) -- (7,\y+1);
							\draw[postaction={decorate}] (7,\y-1) -- (6,\y);
							\draw[postaction={decorate}] (8,\y) -- (7,\y+1);
							\draw[postaction={decorate}] (7,\y-1) -- (8,\y);

							\draw[postaction={decorate}] (10,\y) -- (8,\y);

							\draw[postaction={decorate}] (10,\y) -- (11,\y+1);
							\draw[postaction={decorate}] (11,\y-1) -- (10,\y);
							\draw[postaction={decorate}] (12,\y) -- (11,\y+1);
							\draw[postaction={decorate}] (11,\y-1) -- (12,\y);

							\draw[postaction={decorate}] (2,\y) -- (1,\y);
							\draw[postaction={decorate}] (13,\y) -- (12,\y);

							\draw[postaction={decorate}] (3,\y+1) -- (7,\y+3);
							\draw[postaction={decorate}] (7,\y+1) -- (7,\y+3);
							\draw[postaction={decorate}] (11,\y+1) -- (7,\y+3);

							\draw[postaction={decorate}] (7,\y-3) -- (3,\y-1);
							\draw[postaction={decorate}] (7,\y-3) -- (7,\y-1);
							\draw[postaction={decorate}] (7,\y-3) -- (11,\y-1);

							\fill (3,\y) node {$\alpha_1$};
							\fill (4,\y-2.5) node {$-\alpha_2$};
							\fill (4,\y+2.5) node {$\alpha_1+\alpha_3$};

							\fill (5,\y+1) node {$\alpha_1+\alpha_2$};
							\fill (5,\y-1) node {$-\alpha_3$};

							\fill (7,\y) node {$\alpha_2$};

							\fill (9,\y+1) node {$\alpha_2+\alpha_3$};
							\fill (9,\y-1) node {$-\alpha_1$};

							\fill (11,\y) node {$\alpha_3$};

						}
					\end{scope}
				\end{tikzpicture}
				\caption{The oriented exchange graph of the cyclic quiver with 3 vertices.}
			\end{center}
		\end{figure}
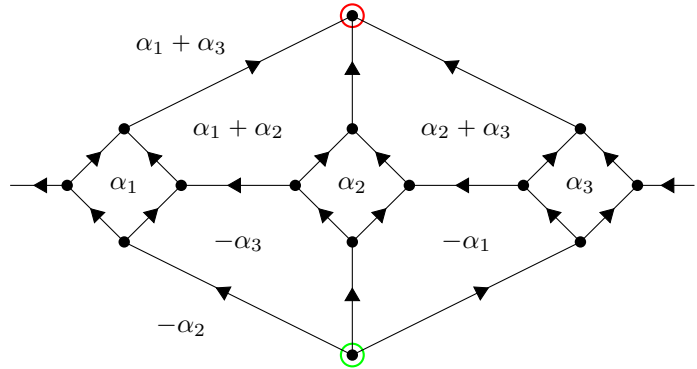
	\end{exmp}

	We refer to Appendix \ref{section:examples} for additional examples.

\section{The infinite cluster type}\label{section:infinite}
	If $Q$ is not of finite cluster type, then $\oo\E(Q)$ is an infinite oriented graph and it is not known whether $\green(Q)$ is a finite set or not. Moreover, we have already seen in Proposition \ref{prop:doubletriangle} that $\green(Q)$ can be empty in the general case. When $Q$ is acyclic, we know from Corollary \ref{corol:lowerbound} that $\green(Q)$ is non-empty so that we will now focus on this case. 

	It is proved in \cite{cluster2} that an acyclic quiver is of finite cluster type if and only if it is an orientation of a Dynkin diagram or, in representation-theoretic terms, if it is of finite representation type. Representation-infinite quivers are partitioned into two sets: \emph{affine quivers}, which are acyclic orientations of extended Dynkin diagrams of types $\widetilde A$, $\widetilde D$ or $\widetilde E$, and \emph{wild quivers}, which are the acyclic quivers which are neither Dynkin nor affine.

	The following lemma will be proved in Section \ref{section:proofsinfinite}.
	\begin{lem}\label{lem:2vertices}
		Let $Q$ be a (connected) cluster quiver with two vertices. Then:
		\begin{enumerate}
			\item either $Q$ is of type $A_2$ and $\green(Q) = 2$, $\ell_{\min}(Q)= 2$ and $\ell_{\max}(Q) = 3$,
			\item or $Q$ is representation-infinite and $\green(Q)=1$ and $\ell_{\min}(Q) = \ell_{\max}(Q) = 2$.
		\end{enumerate}
	\end{lem}

	\subsection{The affine case}
		In the affine case, our main theorem is:
		\begin{theorem}\label{theorem:affine}
			Let $Q$ be an affine quiver. Then $\green(Q)$ is finite and non-empty.
		\end{theorem}

		\begin{exmp}\label{exmp:2bricks}
			Consider the quiver $$
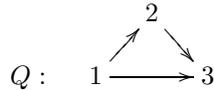
\xymatrix@=1em{ && 2\ar[rd] \\ Q:& 1 \ar[ru] \ar[rr] && 3}$$ of affine type $\widetilde A_{2,1}$. Then locally around $[\h Q]$, the oriented exchange graph $\oo\E(Q)$ can be depicted as follows where the unique source is circled in green and the unique sink is circled in red. Here the faces are labelled by the denominator vectors of the cluster variables in the corresponding clusters, when expressed in the initial seed corresponding to $[\h Q]$.

			
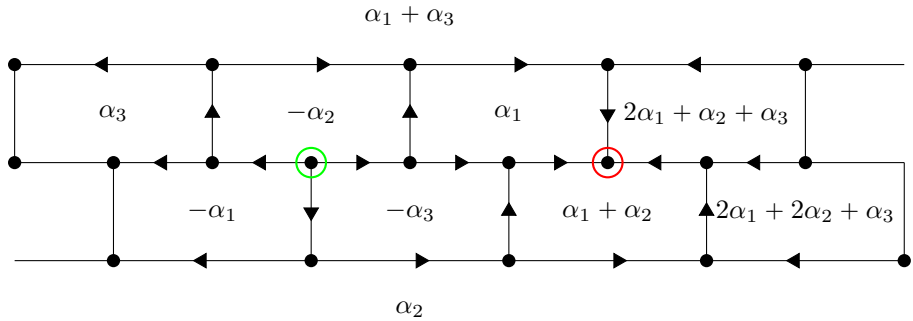
\vspace{1em}

			\begin{center}
				\begin{tikzpicture}[scale = 1.3]
					\begin{scope}[decoration={
								markings,
								mark=at position 0.6 with {\arrow{triangle 60}}}
								] 
						\foreach \y in {0}
						{
							\foreach \x in {-4,-2,...,4}
							{
								\fill (\x,\y+1) circle (.067);
								\fill (\x+1,\y-1) circle (.067);
							}
							\foreach \x in {-4,-3,...,4}
							{
								\fill (\x,\y) circle (.067);
							}
						}

						\draw[postaction={decorate}] (-2,1) -- (-4,1);
						\draw[postaction={decorate}] (-2,1) -- (0,1);
						\draw[postaction={decorate}] (0,1) -- (2,1);
						\draw[postaction={decorate}] (4,1) -- (2,1);

						\draw[postaction={decorate}] (-1,-1) -- (-3,-1);
						\draw[postaction={decorate}] (-1,-1) -- (1,-1);
						\draw[postaction={decorate}] (1,-1) -- (3,-1);
						\draw[postaction={decorate}] (5,-1) -- (3,-1);

						\draw[postaction={decorate}] (-2,0) -- (-2,1);
						\draw[postaction={decorate}] (-1,0) -- (-1,-1);
						\draw[postaction={decorate}] (0,0) -- (0,1);
						\draw[postaction={decorate}] (1,-1) -- (1,0);
						\draw[postaction={decorate}] (2,1) -- (2,0);
						\draw[postaction={decorate}] (3,-1) -- (3,0);

						\draw[postaction={decorate}] (-2,0) -- (-3,0);
						\draw[postaction={decorate}] (-1,0) -- (-2,0);
						\draw[postaction={decorate}] (-1,0) -- (0,0);
						\draw[postaction={decorate}] (0,0) -- (1,0);
						\draw[postaction={decorate}] (1,0) -- (2,0);
						\draw[postaction={decorate}] (3,0) -- (2,0);
						\draw[postaction={decorate}] (4,0) -- (3,0);

						\draw (4,0) -- (5,0);
						\draw (4,1) -- (5,1);
						\draw (-4,-1) -- (-3,-1);
						\draw (-3,0) -- (-4,0);

						\draw (-3,0) -- (-3,-1);
						\draw (-4,1) -- (-4,0);

						\draw (4,0) -- (4,1);
						\draw (5,-1) -- (5,0);

						\draw[thick,green] (-1,0) circle (.15);
						\draw[thick,red] (2,0) circle (.15);

						\fill (0,1.5) node {$\alpha_1 + \alpha_3$};
						\fill (0,-1.5) node {$\alpha_2$};

						\fill (-3,.5) node {$\alpha_3$};
						\fill (-1,.5) node {$-\alpha_2$};
						\fill (1,.5) node {$\alpha_1$};
						\fill (3,.5) node {$2 \alpha_1 + \alpha_2 + \alpha_3$};

						\fill (-2,-.5) node {$-\alpha_1$};
						\fill (0,-.5) node {$-\alpha_3$};
						\fill (2,-.5) node {$\alpha_1 + \alpha_2$};
						\fill (4,-.5) node {$2 \alpha_1 + 2 \alpha_2 + \alpha_3$};
					\end{scope}
				\end{tikzpicture}
			\end{center}
			We see that, in this case, there are exactly five maximal green sequences.

		\end{exmp}

		For additional examples, we refer the reader to Appendix \ref{section:exmpaffines}.

	\subsection{Wild quivers with three vertices}
		For the wild case, the situation appears to be more complicated. It is in fact known that for any (connected) wild quiver $Q$ with at least three vertices, there exist regular tilting $\kQ$-modules \cite{Ringel:ARwild}. Therefore, the proof of Theorem \ref{theorem:affine} cannot be reproduced for wild quivers. However, we will prove in Proposition \ref{prop:noregsgreen} that for quivers with three vertices, regular tilting $\kQ$-modules do not appear along maximal green sequences so that we are still able to deduce the finiteness of $\green(Q)$ in this case.

		\begin{theorem}\label{theorem:3vertices}
			Let $Q$ be an acyclic quiver with three vertices. Then $\green(Q)$ is finite and non-empty.
		\end{theorem}
		The proof is given in Section \ref{section:proofsinfinite}.
		
		If $Q$ is a wild quiver with at least four vertices, we do not know whether $\green(Q)$ is finite or not. In this case, we could only compute a few examples which yield some evidence for the finiteness of this number.

		As previously noticed, if Conjecture \ref{conj:intervals} holds, in order to prove that $\green(Q)$ is a finite set for a given quiver $Q$, it is enough to find the smallest $l \geq 1$ such that $\green_{l}(Q) \neq \emptyset$ and such that $\green_{l+1}(Q) = \emptyset$; in this case $\green(Q) = \green_{\leq l}(Q)$. We now provide an example of a wild quiver for which we computed such integers with the computer program \verb\Quiver Mutation Explorer\ \cite{DupontPerotin:QME}.

		\begin{exmp}\label{exmp:wild1}
			Consider the quiver 
			$$\xymatrix@=1em{
				Q: 1 \ar[r] & 2 \ar[r] & 3 \ar@<-2pt>[r]\ar@<+2pt>[r] & 4.
			}$$
			Then 
			$$\begin{array}{c|c}
				l 	& |\green_l(Q)|\\
				\hline \hline
				4	& 1 \\
				\hline
				5	& 7 \\
				\hline
				6	& 6 \\
				\hline
				7	& 7 \\
				\hline
				8	& 0 \\
			\end{array}$$
			Therefore, $\ell_{\min}(Q) = 4$, $\ell^0_{\max}(Q) = 7$ and $|\green^0(Q)|=21$.
		\end{exmp}

\section{Silting, tilting, cluster-tilting and $t$-structures}\label{section:reptheory}
	As was already mentioned, given a cluster quiver $Q$, the oriented exchange graph $\oo\E(Q)$ we are studying in this article is an orientation of the cluster exchange graph $\E(Q)$ of the cluster algebra $\AA_Q$, which is the dual graph of the cluster complex $\Delta(\AA_Q)$ introduced in \cite{cluster2}. The same exchange graph also arises naturally in representation theory. This was first observed in \cite{BMRRT,CCS1} where it was proved that if $Q$ is an acyclic quiver, then the clusters in $\AA_Q$ correspond bijectively to the cluster-tilting objects in the so-called \emph{cluster category} $\CC_Q$ of $Q$ in such a way that cluster mutations correspond to mutations of cluster-tilting objects in $\CC_Q$. This generalises to arbitrary skew-symmetric cluster algebras by considering the cluster-tilting theory of certain generalised cluster categories, see \cite{Amiot:clustercat,Plamondon:ClusterAlgebras}. The aim of this section is to recall how $\E(Q)$ and $\oo\E(Q)$ arise in the context of additive categorifications and related topics in representation theory.
	
	In the particular case where $Q$ is acyclic, identifying $\modd \kQ$ with a subcategory of the cluster category $\CC_Q$, the tilting $\kQ$-modules become cluster-tilting objects in $\CC_Q$ and therefore, the cluster complex $\Delta(\AA_Q)$ contains a certain subcomplex whose maximal simplices correspond to the tilting $\kQ$-modules. Already in 1987 Ringel observed that the set $\TT_A$ of tilting modules over a finite dimensional algebra $A$ carries the structure of a simplicial complex. The study of this complex and of a poset structure on $\TT_A$ was initiated in \cite{RiedtmannSchofield} and further carried out by Happel and Unger \cite{Unger:3vertices,Unger:simplicial,HU:tiltinghereditary}. We refer to the contributions of Ringel and of Unger in the Handbook of tilting theory for further details \cite{Ringel:handbook,Unger:handbook}.

	In the first part of this section we recall the related notions for tilting modules, and then describe some generalisation to the setup of derived categories. We finally explain the link to cluster categories. 

	Throughout, we fix an algebraically closed field $\k$ and all the algebras we consider are $\k$-algebras. If there is no risk of confusion, for a finite-dimensional algebra $A$, we denote by $\DD = \DD^b(\modd A)$ its bounded derived category with shift functor $[1]$.

	\subsection{Tilting modules and their mutations}
		Let $A$ be a basic connected finite-dimensional $\k$-algebra with $n$ non-isomorphic simple modules.

		\begin{defi}[Tilting modules]
			A finitely generated $A$-module $T$ is called \emph{tilting} if
			\begin{enumerate}
				\item $\pdim T \leq 1$,
				\item $\Ext^i_A(T,T) = 0$ for all $i>0$,
				\item $A$ admits a coresolution in $\modd A$ by $A$-modules in $\add T$.
			\end{enumerate}
		\end{defi}

		A poset structure on the set $\TT_A$ of isomorphism classes of basic tilting modules is defined in \cite{RiedtmannSchofield} by setting
		$$T \leq T' \Leftrightarrow T'^\perp \subseteq T^\perp$$
		where $$T^\perp = \{ X \in \modd A \ | \ \Ext^i(T,X) =0 \text{ for all } i > 0 \}.$$
		We denote by $\oo\KK_{\modd A}$ the Hasse graph of this poset of tilting $A$-modules.
		If $A$ is hereditary, it is shown in \cite{HU:tiltinghereditary} that the unoriented graph underlying this Hasse graph is the dual graph of the complex of tilting $A$-modules: there is an arrow $T \to T'$ in $\oo\KK_{\modd A}$ precisely when $T = \bigoplus_j T_j$ and $T' = \mu_k^+(T) = (T/T_k) \oplus T'_k$ where $T'_k$ is the \emph{forward mutation} of $T$ at some $i$ defined as the cokernel of a minimal left $\add(T/T_k)$-approximation $T_k \to M$ (we usually slightly abuse notations and write $T/T_k$ for $\bigoplus_{j \neq k} T_j$).

		The poset $\TT_A$ has $A$ as unique maximal element, and in case the algebra $A$ is Gorenstein, it has $DA$ as unique minimal element.

	\subsection{Silting objects and their mutations}
		The Hasse graph of the poset of tilting $A$-modules is not $n$-regular since not all tilting modules admit mutations. There are various ways to extend the notion of tilting module to a larger class of objects where mutations are always possible. We refer to \cite{KY:simpleminded} for a more complete picture on those various concepts, and we just recall the concept of silting objects here: 

		Let $\DD$ denote the bounded derived category of $\modd A$ with shift functor $[1]$. 
		\begin{defi}[Silting objects, \cite{KV:aisles}]
			An object $T$ in $\DD$ is called \emph{silting} if:
			\begin{enumerate}
				\item $\Hom_{\DD}(T,T[i])=0$ for any $i>0$, 
				\item $\thick(T) = \DD$
			\end{enumerate}
			where $\thick(T)$ denotes the thick subcategory generated by $T$ in $\DD$.
		\end{defi}

		It is shown in \cite{AI:siltingmutation} that the set $\TT_{\DD}$ of isomorphism classes of basic silting objects is turned into a poset by setting 
		$$T \leq T' \Leftrightarrow T'^\perp \subseteq T^\perp,$$
		where as for modules
		$$ T^\perp = \{ X \in \DD \ | \ \Hom_{\DD}(T,X[i]) =0 \text{ for all } i > 0 \}.$$

		Aihara and Iyama also show in \cite{AI:siltingmutation} that the unoriented graph underlying the Hasse graph $\oo\KK_{{\DD}}$ of $\TT_{\DD}$ is the dual graph of the complex of silting objects in ${\DD}$: there is an arrow $T \to T'$ in $\oo\KK_{{\DD}}$ precisely when $T = \bigoplus_j T_j$ and $T' = \mu_k^+(T) = (T/T_k) \oplus T'_k$ where $T'_k$ is the \emph{forward mutation} of $T$ at some $k$ defined as 
		$$T'_k = \cone ( T_k \to \bigoplus_{j \neq k} \irr(T_k,T_j)^* \otimes T_j).$$

		A \emph{tilting object} in $\DD$ is a silting object $T$ such that $\Hom_{\DD}(T,T[i])=0$ for any $i \neq 0$. In particular, any tilting $A$-module $T$ viewed as a stalk complex in $\DD$ is a tilting object in $\DD$, and therefore a silting object in $\DD$. It follows immediately from the definition that if $T$ and $T'$ are tilting $A$-modules such that $T' = \mu_k^+(T)$ as $A$-modules, then $T' = \mu_k^+(T)$ as silting objects in $\DD$.

	\subsection{$t$-structures and their mutations}
		Let $T$ be a silting object in ${\DD}$ and consider the full subcategories in ${\DD}$:
		$${\DD}^{\leq 0}_T = \ens{N \in {\DD} \ | \ \Hom_{{\DD}}(T,N[i])=0 \text{ for all } i >0}$$
		$${\DD}^{\geq 0}_T = \ens{N \in {\DD} \ | \ \Hom_{{\DD}}(T,N[i])=0 \text{ for all } i <0}.$$

		Then $({\DD}^{\leq 0}_T,{\DD}^{\geq 0}_T)$ is a bounded $t$-structure on ${\DD}$ with \emph{length heart} $\HH_T$ (that is, a heart in which every object has finite length), see for instance \cite{KY:simpleminded}. The simple \emph{forward mutation} (also called \emph{forward tilt}) of a heart of a bounded $t$-structure in ${\DD}$ defined in \cite{HRS:tiltingabelian} corresponds to the forward mutation of the respective silting object in ${\DD}$, see \cite{AI:siltingmutation,KY:simpleminded}. 

%
%
%
%
%
%
%
%
%

		Also from these papers, we summarise the situation as follows: isomorphism classes of basic silting objects in ${\DD}$ correspond bijectively to bounded $t$-structures with length heart in ${\DD}$. The $t$-structures are ordered by inclusion of their left aisles, and the forward mutation describes the arrows in the Hasse graph of these posets, see Figure \ref{fig:fwdmutsilting}.

		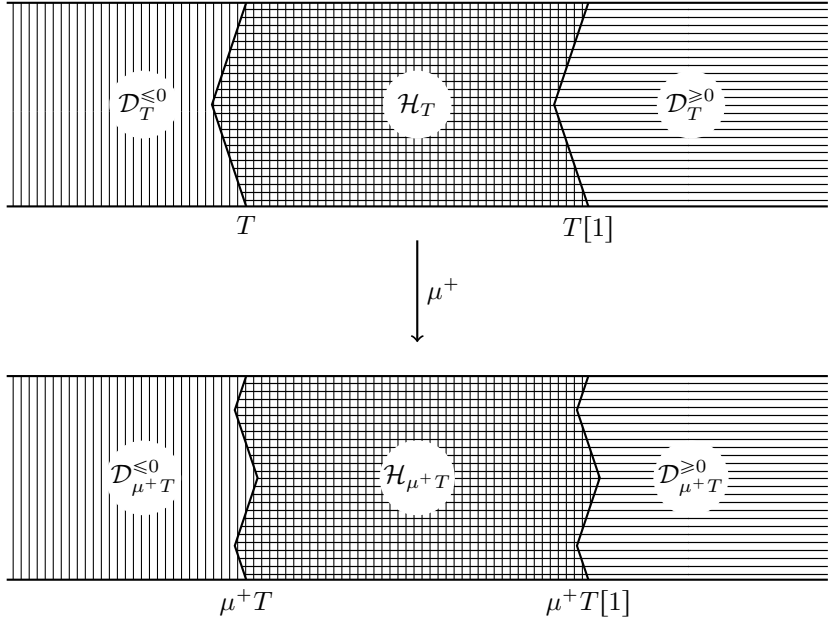
\begin{figure}[htb]
			\begin{center}
				\begin{tikzpicture}[scale = .45]

					\fill[pattern = vertical lines] (-12,3) -- (5,3) -- (4,0) -- (5,-3) -- (-12,-3) -- cycle;
					\fill[pattern = horizontal lines] (12,3) -- (-5,3) -- (-6,0) -- (-5,-3) -- (12,-3) -- cycle;
			
					\draw[thick] (-12,3) -- (12,3);

					\draw[thick] (-5,3) -- (-6,0) -- (-5,-3);
					\draw[thick] (5,3) -- (4,0) -- (5,-3);

					\draw[thick] (-12,-3) -- (12,-3);

					\fill[white] (0,0) circle (1);
					\fill (0,0) node {$\mathcal H_T$};

					\fill[white] (-8,0) circle (1);
					\fill (-8,0) node {$\mathcal D^{\leq 0}_T$};

					\fill[white] (8,0) circle (1);
					\fill (8,0) node {$\mathcal D^{\geq 0}_T$};

					\fill (-5,-3) node [below] {$T$};

					\fill (5,-3) node [below] {$T[1]$};

					\draw[thick,->] (0,-4) -- (0,-7);
					\fill (0,-5.5) node [right] {$\mu^+$};

					\draw[thick] (-12,-8) -- (12,-8);
					\draw[thick] (-12,-14) -- (12,-14);

					\fill[pattern = vertical lines] (-12,-8) -- (5,-8) -- (5-1/3,-9) -- (5+1/3,-11) -- (5-1/3,-13) -- (5,-14) -- (-12,-14) -- cycle;
					\draw[thick] (5,-8) -- (5-1/3,-9) -- (5+1/3,-11) -- (5-1/3,-13) -- (5,-14);

					\fill[pattern = horizontal lines] (12,-8) -- (-5,-8) -- (-5-1/3,-9) -- (-5+1/3,-11) -- (-5-1/3,-13) -- (-5,-14) -- (12,-14) -- cycle;
					\draw[thick] (-5,-8) -- (-5-1/3,-9) -- (-5+1/3,-11) -- (-5-1/3,-13) -- (-5,-14);

					\fill[white] (0,-11) circle (1.1);
					\fill (0,-11) node {$\mathcal H_{\mu^+T}$};

					\fill[white] (-8,-11) circle (1.1);
					\fill (-8,-11) node {$\mathcal D^{\leq 0}_{\mu^+T}$};

					\fill[white] (8,-11) circle (1.1);
					\fill (8,-11) node {$\mathcal D^{\geq 0}_{\mu^+T}$};

					\fill (-5,-14) node [below] {$\mu^+T$};

					\fill (5,-14) node [below] {$\mu^+T[1]$};
				\end{tikzpicture}
			\end{center}
			\caption{Forward mutation of a silting object in $\DD$ and the inclusion of the corresponding left aisles.}\label{fig:fwdmutsilting}
		\end{figure}

		Since there is always an infinite number of silting objects in the derived category, we restrict our study to an interval with maximal element $A$ and minimal element $A[1]$, thus slightly larger than the poset of tilting $A$-modules. We denote by $\oo\E_{\DD}(A,A[1])$ the Hasse graph of the interval formed by the silting objects which are between $A$ and $A[1]$ for this partial order. This interval, which was already considered in \cite{KQ:exchangeCY} appears to be relevant for the purpose of maximal green sequences, dilogarithm identities or for BPS quivers theory \cite{Keller:dilogarithm,CCV:Braids,BD:doublecover}.

	\subsection{Cluster-tilting objects and their mutations}
		\begin{defi}[Cluster-tilting objects]
			A \emph{cluster-tilting object} $T$ in a triangulated category $\CC$ is an object $T$ such that for any $X$ in $\CC$, we have
			$$\Ext^1_{\CC}(T,X) = 0 \Leftrightarrow X \in \add T.$$
		\end{defi}

		Cluster-tilting objects were first considered in \cite{BMRRT} where it was proved that the combinatorics of cluster-tilting objects in cluster categories were governed by mutations in (simply-laced) acyclic cluster algebras.

		Given an acyclic quiver $Q$, its path algebra $\kQ$ is a finite-dimensional hereditary algebra. We denote by $\Gamma$ the \emph{Ginzburg dg-algebra} associated with the quiver with potential $(Q,0)$. It is a 3-Calabi-Yau dg-algebra concentrated in negative degrees, see \cite{Keller:deformed}. We denote by $\DD\Gamma$ the derived category of dg-$\Gamma$-modules, by $\per\Gamma$ its perfect subcategory and by $\DDfd\Gamma$ the full subcategory of $\DD\Gamma$ formed by those dg-modules with finite-dimensional total homology. The \emph{cluster category} of $Q$ is defined in \cite{Amiot:clustercat} as the triangulated quotient $\CC_Q = \per \Gamma/\DDfd\Gamma$. It is a Hom-finite triangulated 2-CY category which is naturally triangle-equivalent to the cluster category defined as an orbit category in \cite{BMRRT}.

		Then $\oo\E(Q) = \oo\E_{\DD \kQ}(\kQ,\kQ[1])$ is an orientation of the graph of mutations of the cluster-tilting objects in $\CC_Q$ and the unique source corresponds to the image of $\Gamma$ under the canonical morphism $\per \Gamma \fl \CC_Q$, see \cite{KellerNicolas:weight} and also \cite{KQ:exchangeCY,KY:simpleminded,Qiu:csortable}.

		For a general cluster quiver $Q$ and a non-degenerate potential $W$ on $Q$, it is still possible to form the triangulated quotient $\CC_{Q,W} = \per \Gamma_{Q,W}/\DDfd\Gamma_{Q,W}$ where $\Gamma_{Q,W}$ is the Ginzburg dg-algebra associated with the quiver with potential $(Q,W)$. Then $\oo\E(Q)$ is an orientation of the connected component of the graph of mutations of cluster-tilting objects in $\CC_{Q,W}$ which contains the image of $\Gamma_{Q,W}$ under the canonical morphism $\per \Gamma_{Q,W} \fl \CC_{Q,W}$. 

		If $\Sigma$ denotes the suspension functor in $\DDfd\Gamma_{Q,W}$, a maximal green sequence for $Q$ corresponds in this context to a sequence of forward mutations from the canonical heart $\HH$ of $\DDfd\Gamma_{Q,W}$ to its shift $\Sigma \HH$, see \cite{Keller:dilogarithm}.

	\subsection{Patterns}\label{ssection:patterns}
		Let $Q$ be a cluster quiver with $n$ vertices and $\mathbb T_n$ denote the $n$-regular tree so that the edges adjacent to any vertex in $\mathbb T_n$ are labelled by $\ens{1, \ldots, n}$. Let $t_0$ be a vertex in that graph. To any vertex $t$ in $\mathbb T_n$ we can associate an ice quiver $Q(t)$ such that $Q(t_0) = \h Q$ and such that $t$ and $t'$ are joined by an edge labelled by $k$ in $\mathbb T_n$ if and only if $Q(t') = \mu_k(Q(t))$. This endows $\mathbb T_n$ with a structure of an oriented graph $\oo{\mathbb T}_n$ by orienting the edge $\xymatrix{ t \ar@{-}[r]^k & t'}$ towards $t'$ if and only if $k$ is green in $Q(t)$.

		Let $W$ be a non-degenerate potential on $Q$ and $\Gamma$ be the corresponding Ginzburg dg-algebra. The category $\DDfd\Gamma$ (with suspension functor $\Sigma$) is endowed with a natural $t$-structure with length heart $\HH$. As explained in \cite{Keller:derivedcluster}, we can associate to any vertex $t$ in $\mathbb T_n$ a heart $\HH(t)$ in $\DDfd\Gamma$ such that $\HH(t_0) = \HH$ and such that there is an arrow $\xymatrix{ t \ar[r]^k & t'}$ in $\oo{\mathbb T}_n$ if and only if $\HH(t')$ is obtained from $\HH(t)$ by a forward mutation at the simple $S_k(t)$ in $\HH(t)$, see also \cite{KQ:exchangeCY,KY:simpleminded}. This is called the \emph{pattern of hearts in $\DDfd\Gamma$}. Then a maximal green sequence $(i_1, \ldots, i_n)$ corresponds to a path $\xymatrix{t_0 \ar[r]^{i_1} & t_1 \ar[r]^{i_2} & \cdots \ar[r]^{i_l} & t_l}$ in $\oo{\mathbb T}_n$ such that $\HH(t_l) \simeq \Sigma \HH(t)$.

		If $Q$ is acyclic, $H=\kQ$ and $\DD= \DD^b(\modd H)$ has shift functor $[1]$, we can also associate to any vertex $t$ in $\mathbb T_n$ a silting object $T(t)$ in $\DD$ such that $T(t_0) = \kQ$ and such that there is an arrow $\xymatrix{ t \ar[r]^k & t'}$ in $\oo{\mathbb T}_n$ if and only if $T(t')$ is obtained from $T(t)$ by a forward mutation at $T^{(t)}_{k}$. This is called the \emph{silting pattern on $\DD$}. Then it follows from \cite{KQ:exchangeCY} that a maximal green sequence $(i_1, \ldots, i_l)$ corresponds to a path $\xymatrix{t_0 \ar[r]^{i_1} & t_1 \ar[r]^{i_2} & \cdots \ar[r]^{i_l} & t_l}$ in $\oo{\mathbb T}_n$ such that $T(t_l) \simeq H[1]$.

\section{More on the Happel-Unger poset}\label{section:HU}
	We assume in this section that $Q$ is a cluster quiver which admits a non-degenerate potential $W$ which is Jacobi-finite, that is, the Jacobian algebra $A=\JJ(Q,W)$ is finite-dimensional. These conditions are clearly satisfied when $Q$ is acyclic (with zero potential so that $\JJ(Q,0)=\k Q$) or when $(Q,W)$ is given by an unpunctured surface, see \cite{ABCP,Labardini:potentialssurfaces}.

	The Jacobian algebra $A$ is Gorenstein \cite{KR:clustertilted}, thus we know from Section \ref{section:reptheory} that the oriented exchange graph $\oo\KK_{\modd A}$ of tilting modules has $A$ as unique maximal element and $DA$ as unique minimal element.
	We note that $\oo\KK_{\modd A}$ is in general not connected, even in cases where the exchange graph of silting objects is connected: Happel and Unger have shown that for an affine acyclic quiver $Q$, the graph $\oo\KK_{\modd A}$ is connected precisely when $Q$ is not of type $\widetilde{A}_{1,s}$ with $s \geq 1$, or $\widetilde{A}_{2,2}$ with alternating orientation \cite{HU:tiltinghereditary}. 

	\begin{theorem}\label{theorem:embedHU}
		Let $Q$ be an acyclic quiver and $H = \kQ$. Then $\oo\KK_{\modd H}$ is a full convex oriented subgraph of $\oo\E(Q)$.
	\end{theorem}
	\begin{proof}
		Let $T$ and $T'$ be two tilting $H$-modules and consider a path 
		$$T \fl T^{(1)} \fl \cdots \fl T^{(l)} \fl T'$$
		in $\oo\E(Q)$. Then we have the following chain of inclusions of the left aisles of the corresponding $t$-structures in $\DD = \DD(\modd H)$:
		$$\DD^{\leq 0}_{T} \subset \DD^{\leq 0}_{T^{(1)}} \subset \cdots \subset \DD^{\leq 0}_{T^{(l)}} \subset \DD^{\leq 0}_{T'}.$$
		If there is some $1 \leq k \leq l$ such that $T^{(k)}$ is not a tilting $H$-module, then it is a silting object in $\oo\E_{\DD}(H,H[1])$ and therefore, it has a summand of the form $P_i[1]$. Thus we have $P_i[1] \in \DD^{\leq 0}_{T^{(k)}}$ and we get $P_i[1] \in \DD^{\leq 0}_{T'}$ which implies 
		$$0 = \Hom_{\DD}(P_i[1],T'[1]) = \Hom_{\DD}(P_i,T') = \Hom_{H}(P_i,T')$$
		so that $T'$ is not sincere, which is a contradiction since every tilting $H$-module is sincere. Therefore, for any $1 \leq k \leq l$, the silting object $T^{(k)}$ is a tilting $H$-module. And as it was already mentioned, the forward mutation of a tilting module $T$ in $\modd H$ coincides with the forward mutation of $T$ viewed as a silting object in $\DD$. Therefore, $\oo\KK_{\modd H}$ is a full convex oriented subgraph of $\oo\E(Q)$.
	\end{proof}

	\begin{exmp}
		Figure \ref{fig:HUA2} shows how the Happel-Unger poset embeds in the oriented exchange graph of type $A_2$. In the HU poset, the unique source is circled in green and the unique sink is circled in red.

		\begin{figure}[H]
			\begin{center}
				\begin{tikzpicture}[scale = .75]
					\begin{scope}[decoration={
								markings,
								mark=at position 0.6 with {\arrow{triangle 60}}}
								] 
						\coordinate (A) at (180+0:2);
						\coordinate (B) at (180+72:2);
						\coordinate (C) at (180+144:2);
						\coordinate (D) at (180+216:2);
						\coordinate (E) at (180+288:2);
						\fill (A) circle (.1);
						\fill (E) circle (.1);

						\fill (A) node [left] {\tiny $P_1 \oplus P_2$ \hspace{1em}};
						\fill (E) node [right] {\tiny \hspace{.5em} $P_1 \oplus I_1$};

						\draw[dashed,postaction={decorate}] (A) -- (B);
						\draw[dashed,postaction={decorate}] (B) -- (C);
						\draw[postaction={decorate}] (A) -- (E);
						\draw[dashed,postaction={decorate}] (E) -- (D);
						\draw[dashed,postaction={decorate}] (D) -- (C);

						\draw[thick,green] (A) circle (.2);
						\draw[thick,red] (E) circle (.2);

					\end{scope}
				\end{tikzpicture}
				\caption{The Happel-Unger poset, sitting in the oriented exchange graph of $1 \fl 2$.}\label{fig:HUA2}
			\end{center}
		\end{figure}
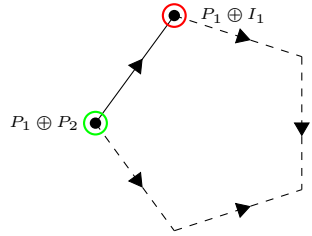
	\end{exmp}

	\begin{exmp}\label{exmp:HUvsgreenA3}
		Figures \ref{fig:HUA3lin} and \ref{fig:HUA3alt} show how the Happel-Unger posets embed in certain oriented exchange graphs of type $A_3$. In both cases, the unique source of the HU poset is circled in green and the unique sink of the HU poset is circled in red.

		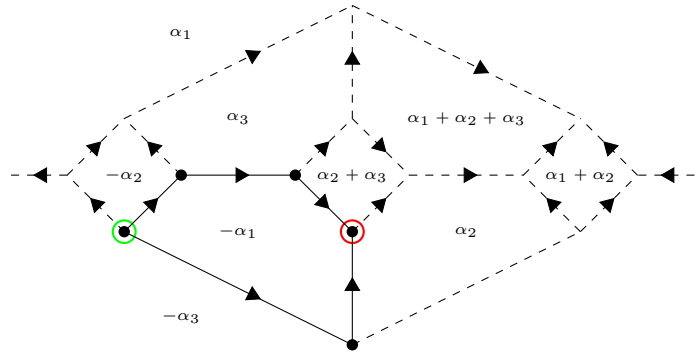
\begin{figure}[H]
			\begin{center}
				\begin{tikzpicture}[scale = .75]
				\tikzstyle{every node}=[font=\tiny]
					\begin{scope}[decoration={
								markings,
								mark=at position 0.6 with {\arrow{triangle 60}}}
								] 
						\foreach \y in {0}
						{

							\fill (3,\y-1) circle (.1);

							\fill (7,\y-1) circle (.1);


							\fill (7,\y-3) circle (.1);

							\foreach \x in {4,6}
							{
								\fill (\x,\y) circle (.1);
							}

							\draw[thick,green] (3,\y-1) circle (.2);
							\draw[thick,red] (7,\y-1) circle (.2);
				
							\draw[dashed,postaction={decorate}] (2,\y) -- (3,\y+1);
							\draw[dashed,postaction={decorate}] (4,\y) -- (3,\y+1);

							\draw[dashed,postaction={decorate}] (3,\y-1) -- (2,\y);
							\draw[postaction={decorate}] (3,\y-1) -- (4,\y);

							\draw[postaction={decorate}] (4,\y) -- (6,\y);

							\draw[dashed,postaction={decorate}] (6,\y) -- (7,\y+1);
							\draw[postaction={decorate}] (6,\y) -- (7,\y-1);
							\draw[dashed,postaction={decorate}] (7,\y+1) -- (8,\y);
							\draw[dashed,postaction={decorate}] (7,\y-1) -- (8,\y);

							\draw[dashed,postaction={decorate}] (8,\y) -- (10,\y);

							\draw[dashed,postaction={decorate}] (10,\y) -- (11,\y+1);
							\draw[dashed,postaction={decorate}] (11,\y-1) -- (10,\y);
							\draw[dashed,postaction={decorate}] (12,\y) -- (11,\y+1);
							\draw[dashed,postaction={decorate}] (11,\y-1) -- (12,\y);

							\draw[dashed,postaction={decorate}] (2,\y) -- (1,\y);
							\draw[dashed,postaction={decorate}] (13,\y) -- (12,\y);

							\draw[dashed,postaction={decorate}] (3,\y+1) -- (7,\y+3);
							\draw[dashed,postaction={decorate}] (7,\y+1) -- (7,\y+3);
							\draw[dashed,postaction={decorate}] (7,\y+3) -- (11,\y+1);

							\draw[postaction={decorate}] (3,\y-1) -- (7,\y-3);
							\draw[postaction={decorate}] (7,\y-3) -- (7,\y-1);
							\draw[dashed] (7,\y-3) -- (11,\y-1);

							\fill (3,\y) node {$-\alpha_2$};
							\fill (4,\y+2.5) node {$\alpha_1$};
							\fill (4,\y-2.5) node {$-\alpha_3$};

							\fill (5,\y+1) node {$\alpha_3$};
							\fill (5,\y-1) node {$-\alpha_1$};

							\fill (7,\y) node {$\alpha_2 + \alpha_3$};

							\fill (9,\y+1) node {$\alpha_1+\alpha_2+\alpha_3$};
							\fill (9,\y-1) node {$\alpha_2$};

							\fill (11,\y) node {$\alpha_1+\alpha_2$};

						}
					\end{scope}
				\end{tikzpicture}
				\caption{The Happel-Unger poset for the quiver $1 \fl 2 \fl 3$, sitting in the poset of maximal green sequences, labelled with denominator vectors.}\label{fig:HUA3lin}
			\end{center}
		\end{figure}

		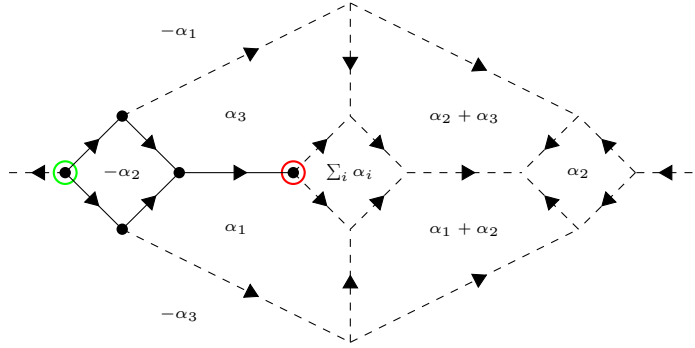
\begin{figure}[H]
			\begin{center}
				\begin{tikzpicture}[scale = .75]
				\tikzstyle{every node}=[font=\tiny]
					\begin{scope}[decoration={
								markings,
								mark=at position 0.6 with {\arrow{triangle 60}}}
								] 
						\foreach \y in {0}
						{

							\fill (3,\y+1) circle (.1);
							\fill (3,\y-1) circle (.1);




							\foreach \x in {2,4,6}
							{
								\fill (\x,\y) circle (.1);
							}

							\draw[thick,green] (2,\y) circle (.2);
							\draw[thick,red] (6,\y) circle (.2);
				
							\draw[postaction={decorate}] (2,\y) -- (3,\y+1);
							\draw[postaction={decorate}] (3,\y+1) -- (4,\y);

							\draw[postaction={decorate}] (2,\y) -- (3,\y-1);
							\draw[postaction={decorate}] (3,\y-1) -- (4,\y);

							\draw[postaction={decorate}] (4,\y) -- (6,\y);

							\draw[dashed,postaction={decorate}] (6,\y) -- (7,\y+1);
							\draw[dashed,postaction={decorate}] (6,\y) -- (7,\y-1);
							\draw[dashed,postaction={decorate}] (7,\y+1) -- (8,\y);
							\draw[dashed,postaction={decorate}] (7,\y-1) -- (8,\y);

							\draw[dashed,postaction={decorate}] (8,\y) -- (10,\y);

							\draw[dashed,postaction={decorate}] (11,\y+1) -- (10,\y);
							\draw[dashed,postaction={decorate}] (11,\y-1) -- (10,\y);
							\draw[dashed,postaction={decorate}] (12,\y) -- (11,\y+1);
							\draw[dashed,postaction={decorate}] (12,\y) -- (11,\y-1);

							\draw[dashed,postaction={decorate}] (2,\y) -- (1,\y);
							\draw[dashed,postaction={decorate}] (13,\y) -- (12,\y);

							\draw[dashed,postaction={decorate}] (3,\y+1) -- (7,\y+3);
							\draw[dashed,postaction={decorate}] (7,\y+3) -- (7,\y+1);
							\draw[dashed,postaction={decorate}] (7,\y+3) -- (11,\y+1);

							\draw[dashed,postaction={decorate}] (3,\y-1) -- (7,\y-3);
							\draw[dashed,postaction={decorate}] (7,\y-3) -- (7,\y-1);
							\draw[dashed,postaction={decorate}] (7,\y-3) -- (11,\y-1);

							\fill (3,\y) node {$-\alpha_2$};
							\fill (4,\y+2.5) node {$-\alpha_1$};
							\fill (4,\y-2.5) node {$-\alpha_3$};

							\fill (5,\y+1) node {$\alpha_3$};
							\fill (5,\y-1) node {$\alpha_1$};

							\fill (7,\y) node {$\sum_i \alpha_i$};

							\fill (9,\y+1) node {$\alpha_2+\alpha_3$};
							\fill (9,\y-1) node {$\alpha_1+\alpha_2$};

							\fill (11,\y) node {$\alpha_2$};

						}
					\end{scope}
				\end{tikzpicture}
				\caption{Happel-Unger poset, sitting in the oriented exchange graph of $1 \lf 2 \fl 3$, labelled with denominator vectors.}\label{fig:HUA3alt}
			\end{center}
		\end{figure}

	\end{exmp}

	\begin{rmq}
		Example \ref{exmp:HUvsgreenA3} shows a phenomenon which was already observed in \cite{HU:tiltinghereditary}, namely that the Happel-Unger poset is not stable under derived equivalences. It also shows that the oriented exchange graph is \emph{not} stable under derived equivalences either. Indeed, in the linear case (Figure \ref{fig:HUA3lin}) there are 9 maximal green sequences whereas there are 10 in the alternating case (Figure \ref{fig:HUA3alt}).
	\end{rmq}

	\begin{lem}\label{lem:longerpaths}
		Let $Q$ be an acyclic quiver and $H = \kQ$. Let $(i_1, \ldots, i_n)$ be an admissible numbering of $Q_0$ by sinks. Assume that there is a path from $H$ to $DH$ in $\oo\KK_{\modd H}$ and let $(v_1, \ldots, v_l)$ be the corresponding green sequence for $Q$. Then $(v_1, \ldots, v_l,i_1,\ldots, i_n) \in \green_{l+n}(Q)$.
	\end{lem}
	\begin{proof}
		Since $DH$ is a tilting $H$-module, it is in particular a tilting object in $\DD(\modd H)$ and therefore a silting object in $\DD = \DD^b(\modd H)$. Let $(i_1, \ldots, i_n)$ be an admissible numbering of $Q_0$ by sinks. The endomorphism algebra of $DH$ has Gabriel quiver $Q^{\op}$ and $(i_1, \ldots, i_n)$ is an admissible sequence of sources for $Q^{\op}$. Therefore, considering successive APR-tilts at sources (see \cite{APR}), we obtain a sequence of tilting objects in $\DD$: 
		$$DH \xrightarrow{\mu_{i_1}} T^{(1)} \xrightarrow{\mu_{i_2}} \cdots \xrightarrow{\mu_{i_{n-1}}} T^{(n-1)} \xrightarrow{\mu_{i_n}} H[1]$$
		where 
		$$T^{(k)} = \left( \bigoplus_{j < k} I_{i_j} \right) \oplus \left( \bigoplus_{l \leq k} \tau^{-1}I_{i_j} \right) = \left( \bigoplus_{j < k} I_{i_l} \right) \oplus \left( \bigoplus_{l \geq k} P_{i_l}[1] \right)$$
		for any $1 \leq k \leq n$. In particular, we have proper inclusions of left aisles
		$$\DD^{\leq 0}_{DH} \subset \DD^{\leq 0}_{T^{(1)}} \subset \DD^{\leq 0}_{T^{(2)}} \subset \cdots \subset \DD^{\leq 0}_{T^{(n-1)}} \subset \DD^{\leq 0}_{H[1]}$$
		so that we obtain a path of length $n$ from $DH$ to $H[1]$ in $\E(Q)$. By Theorem \ref{theorem:embedHU}, a path of length $l$ from $H$ to $DH$ in $\oo\KK_{\modd H}$ gives rise to a path of length $l$ from $H$ to $DH$ in $\oo\E(Q)$, composing this path with the above path from $DH$ to $H[1]$ gives a path from $H$ to $H[1]$ of length $n+l$ in $\oo\E(Q)$, and therefore an element in $\green_{l+n}(Q)$. Figure \ref{fig:longerpaths} illustrates the proof.

		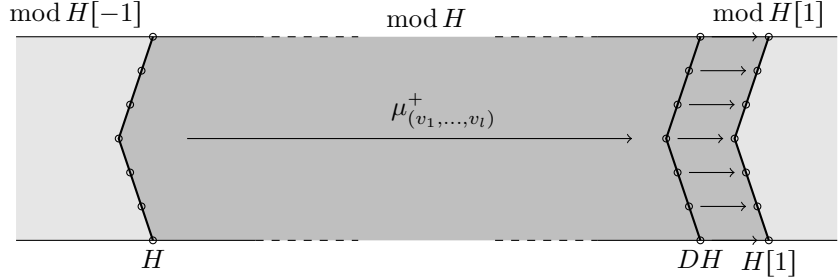
\begin{figure}[htb]
			\begin{center}
				\begin{tikzpicture}[scale = .45]
					\fill[gray!20] (2,-3) -- (1,0) -- (2,3) -- (-2,3) -- (-2,-3) -- cycle;

					\fill[gray!50] (2,-3) -- (1,0) -- (2,3) -- (20,3) -- (19,0) -- (20,-3) -- cycle;

					\fill[gray!20] (22,-3) -- (22,3) -- (20,3) -- (19,0) -- (20,-3) -- cycle;

					\draw (-2,3) -- (5,3);
					\draw[dashed] (5,3) -- (8,3);
					\draw[dashed] (12,3) -- (15,3);
					\draw (15,3) -- (22,3);

					\draw[thick] (2,3) -- (1,0) -- (2,-3);
					\draw (2,3) circle (.1);
					\draw (2-1/3,2) circle (.1);
					\draw (2-2/3,1) circle (.1);
					\draw (1,0) circle (.1);
					\draw (2,-3) circle (.1);
					\draw (2-1/3,-2) circle (.1);
					\draw (2-2/3,-1) circle (.1);

					\draw (18,3) circle (.1);
					\draw (18-1/3,2) circle (.1);
					\draw (18-2/3,1) circle (.1);
					\draw (17,0) circle (.1);
					\draw (18,-3) circle (.1);
					\draw (18-1/3,-2) circle (.1);
					\draw (18-2/3,-1) circle (.1);

					\draw (20,3) circle (.1);
					\draw (20-1/3,2) circle (.1);
					\draw (20-2/3,1) circle (.1);
					\draw (19,0) circle (.1);
					\draw (20,-3) circle (.1);
					\draw (20-1/3,-2) circle (.1);
					\draw (20-2/3,-1) circle (.1);

					\draw[thick] (18,3) -- (17,0) -- (18,-3);

					\draw[thick] (20,3) -- (19,0) -- (20,-3);

					\draw (-2,-3) -- (5,-3);
					\draw[dashed] (5,-3) -- (8,-3);
					\draw[dashed] (12,-3) -- (15,-3);
					\draw (15,-3) -- (22,-3);

					\draw[->] (3,0) -- (16,0);
					\fill (10.5,0) node [above] {$\mu_{(v_1, \ldots, v_l)}^+$};

					\draw[->] (18+1/3,3) -- (20-1/3,3);
					\draw[->] (18,2) -- (20-2/3,2);
					\draw[->] (18-1/3,1) -- (20-1,1);
					\draw[->] (18-2/3,0) -- (20-4/3,0);
					\draw[->] (18-1/3,-1) -- (20-1,-1);
					\draw[->] (18,-2) -- (20-2/3,-2);
					\draw[->] (18+1/3,-3) -- (20-1/3,-3);

					\fill (2,-3) node [below] {$H$};

					\fill (18,-3) node [below] {$DH$};

					\fill (20,-3) node [below] {$H[1]$};

					\fill (-.25,3) node [above] {$\modd H[-1]$};

					\fill (10,3) node [above] {$\modd H$};

					\fill (20,3) node [above] {$\modd H[1]$};
				\end{tikzpicture}
			\end{center}
			\caption{Extending a path from $H$ to $DH$ to a maximal green sequence.}\label{fig:longerpaths}
		\end{figure}
	\end{proof}

	The statement of Lemma \ref{lem:longerpaths} fails if $Q$ is not acyclic: in case $Q$ is the 3-cycle with non-degenerate potential that 3-cycle, the Jacobian algebra $A$ is self-injective and so the poset $\oo\KK_{\modd A}$ consists only of one point. The minimal length of a maximal green sequence is 4, thus the statement of Lemma \ref{lem:longerpaths} does not hold. 

	\begin{rmq}
		Lemma \ref{lem:longerpaths} provides a criterion for the non-existence of paths from $H$ to $DH$ in $\oo\KK_{\modd H}$. For instance, consider the quiver
		$$\xymatrix@=1em{
			Q: 1 \ar[r] & 2 \ar[r] & 3 \ar@<-2pt>[r]\ar@<+2pt>[r] & 4.
		}
		$$
		Since $\k Q$ is wild hereditary, it has no projective-injective modules. Therefore, a path from $\k Q$ to $D\k Q$ in $\oo\KK_{\modd \k Q}$ must have a length at least 4. However, as we saw in Example \ref{exmp:wild1}, we have $\ell^0_{\max}(Q_1) = 7$. Therefore, if Conjecture \ref{conj:intervals} holds, then $\green_{l}(Q)$ is empty for $l \geq 8$ and there is no path from $\k Q$ to $D\k Q$ in $\oo\KK_{\modd \k Q}$.
	\end{rmq}

\section{Proofs of Section \ref{section:greenseq}}\label{section:proofsgreenseq}
	\subsection{Proof of Proposition \ref{prop:uniquered}}
		\begin{proof}
			Let $W$ be a generic potential on $Q$ and let $\Gamma$ be the Ginzburg dg-algebra associated with the quiver with potential $(Q, W)$. The category $\DDfd\Gamma$ is endowed with a natural $t$-structure and we denote by $\HH$ the corresponding heart with simples $S_i$, $i \in Q_0$. Let $\mathbb T_{|Q_0|}$ denote the $|Q_0|$-regular tree and consider the pattern of tilts $t \mapsto \HH(t)$ with $\HH(T_0) = \HH$, see Section \ref{ssection:patterns} or \cite[\S 7.7]{Keller:derivedcluster}. 

			Since $R \in \Mut(\h Q)$, there is some vertex $t$ in $\mathbb T_{|Q_0|}$ such that $R = Q(t)$ and since all the non-frozen vertices in $R$ are green, it means that all the simples in $\HH(t)$ are in $\HH$. Therefore, there exists a permutation $\pi \in \mathfrak S_{Q_0}$ such that $S_i(t) \simeq S_{\pi(i)}$ for any $i \in Q_0$ and it follows from \cite[Corollary 7.11]{Keller:derivedcluster} that $\pi$ induces the wanted isomorphism of ice quivers.

			Similarly, if all the non-frozen vertices in $R=Q(t)$ are red, it means that all the simples in $\HH(t)$ are in $\Sigma\HH$. Therefore, there exists a permutation $\pi \in \mathfrak S_{Q_0}$ such that $S_i(t) \simeq \Sigma S_{\pi(i)}$ for any $i \in Q_0$ and it follows from \cite[Corollary 7.11]{Keller:derivedcluster} that $\pi$ induces the wanted isomorphism of ice quivers.
		\end{proof}

	\subsection{Proof of Proposition \ref{prop:doubletriangle}}
		\begin{proof}
			We know from Plamondon's thesis \cite[Example 4.3]{Plamondon:generic} that there exists a (Jacobi-finite) non-degenerate potential $W$ on $Q$ such that there is no sequence of mutations in the cluster category $\CC_{Q,W}$ joining the cluster-tilting object $\Gamma_{Q,W}$ to the cluster-tilting object $\Sigma \Gamma_{Q,W}$. Therefore, $\Gamma_{Q,W}$ and $\Sigma \Gamma_{Q,W}$ are in two different connected components of the mutation graph of cluster-tilting objects in $\CC_{Q,W}$. In particular, if $\HH$ denotes the canonical heart in $\DDfd\Gamma_{Q,W}$, there is no sequence of forward mutations from $\HH$ to $\Sigma \HH$ and therefore, there is no path from $[\h Q]$ to $[\v Q]$ in $\oo\E(Q)$. In other words, $\green(Q) = \emptyset$.
		\end{proof}

	\subsection{On Jacobi-infinite quivers with potential}
		In this short section we give a criterion for the non-existence of maximal green sequences. We recall that a quiver with potential is called Jacobi-infinite if the corresponding (completed) Jacobian algebra is infinite dimensional over $\k$.

		\begin{prop}\label{prop:nogreeninfinite}
			Let $(Q,W)$ be a Jacobi-infinite quiver with potential. Assume that it is non-degenerate. Then $\green(Q) = \emptyset$.
		\end{prop}
		\begin{proof}
			Let $\Gamma$ be the Ginzburg dg-algebra corresponding to $(Q,W)$. Let $\HH$ denote the canonical heart in $\DDfd\Gamma$. We claim that $\Sigma \HH$ is not reachable from $\HH$ by iterated forward mutations of hearts in $\DDfd\Gamma$. Indeed, if so, we would obtain a sequence of mutations from $\Gamma$ to $\Sigma\Gamma$ in the generalised cluster category $\CC$ associated with $(Q,W)$, see \cite{Keller:derivedcluster,Plamondon:ClusterAlgebras}. Therefore, $\Gamma$ and $\Sigma\Gamma$ are in the same connected component of the cluster-tilting graph of $\CC$. Fix thus a sequence $\i=(i_1, \ldots, i_l)$ such that $\Gamma = \mu_\i(\Sigma\Gamma)$. Then it follows from \cite{Plamondon:ClusterAlgebras}, that this gives a sequence of mutations of decorated representation in the sense of \cite{DWZ:potentials2} from the decorated representation of $(Q,W)$ corresponding to $\Hom_{\CC}(\Gamma,\Sigma \Gamma)=0$ to the decorated representation of $\mu_{\i}(Q,W)$ of $\Hom_{\CC}(\Gamma,\Gamma) \simeq \End_{\CC}(\Gamma) \simeq \JJ(Q,W)$. However, $\JJ(Q,W)$ is infinite dimensional by hypothesis so that it is not a decorated representation in the sense of \cite{DWZ:potentials2}, a contradiction. Thus, there is no sequence of forward mutations from $\HH$ to $\Sigma \HH$ and therefore, there is no path from $[\h Q]$ to $[\v Q]$ in $\oo\E(Q)$. Hence, $\green(Q) = \emptyset$.
		\end{proof}

		\begin{exmp}\label{exmp:VdB}
			Consider the McKay quiver
			$$\xymatrix@=1em{
							&&& 0 \ar[lldd] \ar@<+2pt>[rdddd] \ar@<-2pt>[rdddd]\\ \\
				Q:	& 1 \ar[rdd] \ar@<+2pt>[rrrr]\ar@<-2pt>[rrrr] &&&& 4. \ar[lluu] \ar@<+2pt>[ddlll] \ar@<-2pt>[ddlll]\\ \\
						&& 2 \ar[rr] \ar@<+2pt>[ruuuu]\ar@<-2pt>[ruuuu] && 3 \ar[ruu] \ar@<+2pt>[llluu]\ar@<-2pt>[llluu]
			}
			$$
			Then the main theorem of \cite{VVdB:nondegenerateqps} asserts that $Q$ admits a non-degenerate potential such that the corresponding Jacobian algebra is infinite dimensional. Therefore, it follows from Proposition \ref{prop:nogreeninfinite} that $Q$ has no maximal green sequences.
		\end{exmp}

\section{Proofs of Section \ref{section:finite}}\label{section:proofsfinite}

	\subsection{Proof of Theorem \ref{theorem:Dynkinminmax}}
		\begin{proof}
			A Dynkin quiver is acyclic so that the first point follows from Lemma \ref{lem:lminacyclic}. 

			For the second point, consider the Weyl group $W$ associated with $Q$ with simple reflections $s_i$, with $i \in Q_0$. Let $w_0$ be the longest element in $W$ and fix a reduced expression $w_0 = s_{i_1}\cdots s_{i_r}$ so that $r = |\Phi_+(Q)|$. Then it is well-known that one can choose $w_0$ in such a way that $\i = (i_1, \ldots, i_r)$ is an admissible sequence of sinks in $Q$. For any $1 \leq k \leq r$, we set 
			$$T^{(k)} = \mu_{i_k}^+ \circ \cdots \circ \mu_{i_1}^+(H)$$
			with the convention that $T^{(0)} = H$.

			Since $\i$ is an admissible sequence of sinks, for any $1 \leq k \leq r$, the vertex $i_k$ is a sink in the quiver of the endomorphism ring of $T^{(k-1)}$ so that $T^{(k)}$ is obtained from $T^{(k-1)}$ by a simple APR-tilt (see \cite{APR}). Therefore, the left aisles $\DD^{\leq 0}_{T^{(k-1)}}$ and $\DD^{\leq 0}_{T^{(k)}}$ differ by a single indecomposable object, namely $T_{i_k}^{(k-1)}$. Moreover, it is well-known that $T^{(r)} \simeq H[1]$. Therefore, we obtained a sequence of forward mutations
			$$H \xrightarrow{\mu_{i_1}^+} T^{(1)} \xrightarrow{\mu_{i_2}^+} \cdots \xrightarrow{\mu_{i_r}^+} T^{(r)} \simeq H[1]$$
			which is the longest possible. Thus $\i$ is a maximal green sequence of the longest possible length and we have $\ell_{\max}(Q) = r = |\Phi_+(Q)|$.
		\end{proof}

		Note that another proof of this result can be found in \cite[Proposition 6.3]{Qiu:Dynkin}.

\section{Proofs of Section \ref{section:infinite}}\label{section:proofsinfinite}
	\subsection{Proof of Lemma \ref{lem:2vertices}}
		\begin{proof}
			The first point follows from Theorem \ref{theorem:Dynkinminmax}. We now prove the second point. Let $(i_1,i_2)$ be an admissible numbering of $Q_0$ by sources. Then it was proved in Lemma \ref{lem:lminacyclic} that $(i_1,i_2)$ is a maximal green sequence for $Q$. If there exists another maximal green sequence, then it is necessarily obtained by iterated mutations at sinks of the form $i_2i_1i_2i_1\cdots$. Let $H = \kQ$, $T^{(0)} = H$ and for any $k \geq 1$, let 
			$$T^{(k)} = \left\{\begin{array}{ll}
				\mu^+_{i_2}(T^{(k-1)}) & \text{ if $k$ is odd,}\\
				\mu^+_{i_1}(T^{(k-1)}) & \text{ if $k$ is even.}
			\end{array}\right.$$
			Then, as in the proof of Theorem \ref{theorem:Dynkinminmax}, the left aisles $\DD^{\leq 0}_{T^{(k-1)}}$ and $\DD^{\leq 0}_{T^{(k)}}$ differ by a single indecomposable object, namely $T_{i_k}^{(k-1)}$. However, the left aisle $\DD^{\leq 0}_{H[1]}$ contains infinitely many more objects than the left aisle $\DD^{\leq 0}_H$ so that $T^{(k)} \not \simeq H[1]$ for any $k \geq 1$. Therefore, there is no maximal green sequence beginning with $i_2$ and thus $\green(Q) = \ens{(i_1,i_2)}$.
		\end{proof}

	\subsection{Proof of Theorem \ref{theorem:affine}}
		Let $Q$ be an affine quiver and let $H = \kQ$. The aim of this subsection is to prove that $\green(Q)$ is a finite set. In fact, there are infinitely many green sequences starting at the silting object $H$, and we have to show that only finitely many of them can terminate at $H[1]$. To do that, we first recall some results from representation theory.

		We let $S_1, \ldots, S_n$ denote the simple $H$-modules and for any $1 \leq i \leq n$, we denote by $P_i$ the projective cover of $S_i$ and by $I_i$ its injective hull.

		As usual, we let $\DD$ be the bounded derived category of $\modd H$ with shift functor $[1]$. We denote by $\Gamma(\DD)$ the Auslander-Reiten quiver of $\DD$, by $\PP$ the connected component containing the projective $H$-modules and by $\II$ the connected component of $\Gamma(\DD)$ containing the injective $H$-modules. 
		The indecomposable $H$-modules which are neither in $\PP$ nor in $\II$ are referred to as regular modules.

		We start with a general lemma:
		\begin{lem}\label{lem:tausincere}
			Let $H$ be a representation-infinite connected hereditary algebra. Then there exists $N \geq 0$ such that for any $k \geq N$, for any projective $H$-module $P$ and for any injective $H$-module $I$, the $H$-modules $\tau^{-k} P$ and $\tau^{k} I$ are sincere.
		\end{lem}
		\begin{proof}
			For any $1 \leq i \leq n$, it is known that the sets $\ens{\tau^{-k}P_i}_{k \geq 0}$ and $\ens{\tau^{k}I_i}_{k \geq 0}$ contain only finitely many non-sincere modules, see \cite[Ch. IX, Proposition 5.6]{ASS}. Therefore, there exists $N_i \geq 0$ such that $\tau^{-k}P_i$ and $\tau^{k}I_i$ are sincere for any $k \geq N_i$. Then $N = \max \ens{N_i \ | \ 1 \leq i \leq n}$ is as wanted.
		\end{proof}

		\textbf{Proof of Theorem \ref{theorem:affine}:}
		Let $Q$ be an affine quiver and $H = \kQ$. Then a maximal green sequence for $Q$ is an oriented path 
		$$H = T^{(0)} \fl T^{(1)} \fl \cdots \fl T^{(p-1)} \fl T^{(p)} = H[1]$$ in $\oo\KK_\DD$ from $H$ to $H[1]$. 
		Note that $H$ has its indecomposable summands in $\PP$ and $H[1]$ has all of its indecomposable summands in $\II$.
		
		Since the quiver $Q$ is affine there are only finitely many indecomposable rigid regular $H$-modules, and so the number of isomorphism classes of regular indecomposable summands of silting object arising on an oriented path from $H$ to $H[1]$ in $\oo\KK_\DD$ is finite. We want to show that the same holds for $\PP$ and $\II$.
		
		Let $N$ be as in Lemma \ref{lem:tausincere}, and suppose that one of the silting objects $T^{(s)}$ along the oriented path contains a summand of the form $\tau^{-k} P_i $ for an indecomposable projective $H$-module $P_i$ and $k \ge 2N$. 
		Every possible summand lying in $\II$ can be written as $\tau^{l} P_j[1]$ for some $l \ge 0$ and some indecomposable projective $H$-module $P_j$, and we have
		
		\begin{align*}
			\Ext^1_\DD(\tau^{l} P_j[1], \tau^{-k}P_i)
				& \simeq D\Hom_\DD(\tau^{-k}P_i, \tau^{l+1} P_j[1]) \\
				& \simeq D\Hom_\DD(\tau^{-k-l}P_i, I_j).
		\end{align*}
		But the space $\Hom_\DD(\tau^{-k-l}P_i, I_j)$ is non-zero since $\tau^{-k-l}P_i$ is sincere by Lemma \ref{lem:tausincere} and a sincere module maps nontrivially into every injective $I_j$. 
		This shows that $T^{(s)}$ cannot contain both a summand of the form $\tau^{-k} P_i $ and a summand lying in $\II$.
		Moreover, it also follows from Lemma \ref{lem:tausincere} that $\Ext^1_\DD(\tau^{-m}P_j, \tau^{-k}P_i) \neq 0$ for any $|m-k| > N$, thus all other possible summands of $T^{(s)}$ lying in $\PP$ are of the form $ \tau^{-m}P_j$ for some $N < m $ since we have chosen $k \ge 2N$. By the same argument as above, such a summand would also extend with any object in $\II$.
		
		Since $H$ is tame, any tilting object in $\DD$ has at most $n - 2$ indecomposable regular modules as direct summands, see \cite{Ringel:1099}. The same holds for any silting object $S$ such that $H \leq S \leq H[1]$. Thus, $T^{(s)}$ has at least two direct summands in $\PP$, and therefore also $T^{(s+1)}$ cannot contain a summand in $\II$ since it still contains one of the two summands of $T^{(s)}$ lying in $\PP$. This means the path would never terminate at $H[1]$, a contradiction. 
		This proves that the number of isomorphism classes of indecomposable summands in $\PP$ of a silting object arising on an oriented path from $H$ to $H[1]$ in $\oo\KK_\DD$ is finite.
		The proof of the same statement for $\II$ is dual. \hfill \qed

		\begin{figure}[htb]
			\begin{center}
				\begin{tikzpicture}[scale = .4]
					\fill[gray!40] (-10,3) -- (-9,0) -- (-10,-3) -- (-13,-3) -- (-12,0) -- (-13,3) -- cycle;
					\fill[gray!40] (9,3) -- (10,0) -- (9,-3) -- (12,-3) -- (13,0) -- (12,3) -- cycle;

					\draw (-15,3) -- (-8,3);
					\draw[dashed] (-8,3) -- (-7,3);
					\draw (-15,-3) -- (-8,-3);
					\draw[dashed] (-8,-3) -- (-7,-3);

					\draw (15,3) -- (8,3);
					\draw[dashed] (8,3) -- (7,3);
					\draw (15,-3) -- (8,-3);
					\draw[dashed] (8,-3) -- (7,-3);

					\foreach \t in {-4,0,4,8}
					{
						\foreach \x in {-2.5,-2.3,...,-1.5}
						{
							\draw (\t+\x,-4) -- (\t+\x,4);
							\draw[dashed] (\t+\x,4) -- (\t+\x,5);
							\fill (\t+\x,-4) circle (.05);
						}
					}

					\foreach \x in {0,4,8}
					{

						\fill[gray!40] (-5+\x,-4) .. controls (-5+\x,-4.5) and (-3+\x,-4.5) .. (-3+\x,-4) -- (-3+\x,-2) .. controls (-3+\x,-2.5) and (-5+\x,-2.5) .. (-5+\x,-2) -- cycle;
						\fill[gray!20] (-3+\x,-2) .. controls (-3+\x,-2.5) and (-5+\x,-2.5) .. (-5+\x,-2) .. controls (-5+\x,-1.5) and (-3+\x,-1.5) .. (-3+\x,-2);

						\draw (-3+\x,-2) .. controls (-3+\x,-2.5) and (-5+\x,-2.5) .. (-5+\x,-2);
						\draw[dashed] (-3+\x,-2) .. controls (-3+\x,-1.5) and (-5+\x,-1.5) .. (-5+\x,-2);

						\draw (-5+\x,-4) -- (-5+\x,4);
						\draw (-3+\x,-4) -- (-3+\x,4);

						\draw (-5+\x,-4) .. controls (-5+\x,-4.5) and (-3+\x,-4.5) .. (-3+\x,-4);
						\draw[dashed] (-5+\x,-4) .. controls (-5+\x,-3.5) and (-3+\x,-3.5) .. (-3+\x,-4);
						\draw[dashed] (-5+\x,4) .. controls (-5+\x,4.5) and (-3+\x,4.5) .. (-3+\x,4);
						\draw[] (-5+\x,4) .. controls (-5+\x,3.5) and (-3+\x,3.5) .. (-3+\x,4);
						\draw[dashed] (-5+\x,4) -- (-5+\x,5);
						\draw[dashed] (-3+\x,4) -- (-3+\x,5);

						\draw[dashed] (-3+\x,4) -- (-3+\x,5);
					}

					\draw (-10,3) -- (-9,0) -- (-10,-3);
					\draw[thick] (-13,3) -- (-12,0) -- (-13,-3);
					\fill (-13,-3) node [below] {$H$};

					\draw (9,3) -- (10,0) -- (9,-3);
					\draw[thick] (12,3) -- (13,0) -- (12,-3);
					\fill (12,-3) node [below] {$H[1]$};					 
				\end{tikzpicture}
			\end{center}
			\caption{Schematic picture of the Auslander-Reiten quiver of $\DD^b(\modd H)$ for a tame hereditary algebra. Shaded areas correspond to where silting objects located on a path from $H$ to $H[1]$ can have their indecomposable summands.}\label{fig:siltingaffine}
		\end{figure}
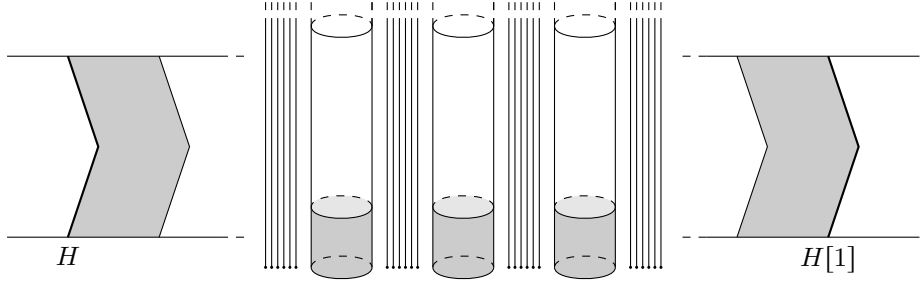		

	\subsection{Proof of Theorem \ref{theorem:3vertices}}
		In this section, we want to prove that $\green(Q)$ is a finite set for an acyclic quiver with 3 vertices.

		We first show the following proposition.
		\begin{prop}\label{prop:noregsgreen}
			Let $Q$ be a connected wild quiver with three vertices and $H=\kQ$. Assume that $T$ is a silting object arising on a path from $H$ to $H[1]$ in $\oo\E(Q)$. Then $T$ has at most one indecomposable regular summands.
		\end{prop}
		\begin{proof}
			We first prove that no regular $H$-module can appear along a maximal green sequence. Assume that $T'$ is the first regular tilting $H$-module arising on a given maximal green sequence. Therefore, this sequence contains an arrow $T \xrightarrow{\mu_v} T'$ where $T' = T/T_v \oplus T_v^*$ with $T_v$ preprojective, say $T_v \simeq \tau^{-s}P_j$, and $T_v^*$, $T/T_v$ regular. 

			Since $T$ is tilting, we get
			\begin{align*}
				0 
					& = \Ext^1_{\DD}(T/T_v,T_v) \\
					& \simeq D \Hom_\DD(T_v,\tau (T/T_v)) \\
					& \simeq D \Hom_\DD(\tau^{-s} P_j, \tau (T/T_v)) \\
					& \simeq D \Hom_\DD(P_j, \tau^{s+1} (T/T_v)) \\
					& \simeq D \Hom_H(P_j, \tau^{s+1} (T/T_v)).
			\end{align*}
			Therefore, $\tau^{s+1}(T/T_v)$ is an $A/(e_j)$-module which is rigid (since $\tau^{s+1} (T/T_v)$ is rigid as an $A$-module) and, since it has $|Q_0|-1$ indecomposable summands, it is tilting as a $A/(e_j)$-module and thus, $\tau^{s+1}T' = \tau^{s+1}(T/T_v) \oplus \tau^{s+1} T_v^*$ is a regular tilting $H$-module satisfying the hypothesis of \cite[Theorem 4.3]{Unger:3vertices}. Hence, any tilting module in the same connected component of $\oo\KK_{\modd H}$ as $\tau^{s+1} T'$ contains at least two $\tau$-sincere indecomposable summands (that is, every module in their $\tau$-orbits is sincere). These indecomposables summands are in particular regular $H$-modules. In particular, any predecessor in $\oo\KK_{\modd H}$ of $\tau^{s+1}T'$ has two regular summands and thefore, any predecessor in $\oo\KK_{\modd H}$ of $T'$ has two regular summands, a contradiction. By convexity of $\oo\KK_{\modd H}$ inside $\oo\E(Q)$ (Theorem \ref{theorem:embedHU}), there is no path from $H$ to $T'$ in $\oo\E(Q)$. 
			
			Assume now that $T$ is a silting object on a maximal green sequence of the form $T = T/T_v \oplus T_v$ where $T/T_v$ is regular. By the previous discussion, $T_v$ cannot be regular. Assume that it is preprojective and set $T_v = \tau^{-s} P_j$ for some $s \geq 0$ and $j \in Q_0$ (the proof is dual for $T_v$ preinjective). If $v$ corresponds to a green vertex in $T$, there is an arrow $T \fl T' = T/T_v \oplus T_v^*$ in $\oo\E(Q)$. Now, as in the previous discussion, any predecessor in $\oo\KK_{\modd H}$ of $\tau^{s+1}T'$ has two regular summands so that any predecessor in $\oo\KK_{\modd H}$ of $T'$, and therefore of $T$, has two regular summands, a contradiction. Dually, if $v$ corresponds to a red vertex in $T$, there is an arrow $T' \fl T$ in $\oo\E(Q)$. As in the previous discussion, any successor in $\oo\KK_{\modd H}$ of $\tau^{s+1}T'$ has two regular summands so that any successor in $\oo\KK_{\modd H}$ of $T'$, and therefore of $T$, has two regular summands, a contradiction. By convexity of $\oo\KK_{\modd H}$ inside $\oo\E(Q)$, there is no path from $H$ to $T$.
		\end{proof}

		\textbf{Proof of Theorem \ref{theorem:3vertices}:}
			Without loss of generality we can restrict to the case where $Q$ is connected. If it is Dynkin or affine, the result is known, see Theorems \ref{theorem:finitetype} and \ref{theorem:affine}. We can thus restrict to the case where $Q$ is wild. We let $H = \kQ$. The number of maximal green sequences for $Q$ equals the number of paths from $H$ to $H[1]$ in $\oo\E(Q)$. 

			Consider such a path and let $T$ be the last silting object along this path which is without summand of the form $\tau^l I_i$, with $l \geq -1$ and $i \in Q_0$. According to Proposition \ref{prop:noregsgreen}, $T$ has at most one regular direct summand. Dually, the first silting object without preprojective summand contains at most one regular direct summand. Then it follows from Lemma \ref{lem:tausincere} that the non-regular summands of silting objects between $H$ and $H[1]$ run over a finite set of isomorphism classes.
			
			Thus, if we consider a maximal green sequence, at the step before a regular summand possibly appears, the other two summands are necessarily preprojective or preinjective according to Proposition \ref{prop:noregsgreen}. In particular, it follows from Lemma \ref{lem:tausincere} that there are only finitely many possible configurations for those two summands, and therefore a finite number of possible isomorphism classes of regular summands. Therefore, only finitely many isomorphism classes of silting objects can appear along a maximal green sequence and so that the number of maximal green sequences is finite. \hfill \qed

\section*{Acknowledgements}
The authors would like to thank Bernhard Keller and Dong Yang for interesting discussions on the topic. They would also like to thank an anonymous referee for valuable comments and corrections. The first author acknowledges support from NSERC and Bishop's University. The second author acknowledges support from the ANR \emph{G\'eom\'etrie Tropicale et Alg\`ebres Amass\'ees}.
		
\appendix 

\section{Examples}\label{section:examples}
	\subsection{Rank two oriented exchange graphs}\label{ssection:rank2}
		Any connected valued quiver with two vertices is either of infinite type or of type $A_2$, $B_2$, $C_2$ or $G_2$. We list below the corresponding oriented exchange graphs. It is interesting to observe that in the non-simply-laced types, the lengths of maximal green sequences do not form an interval.

		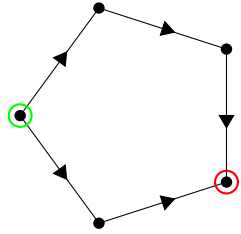
\begin{figure}[H]
			\begin{center}
				\begin{tikzpicture}[scale = .75]
					\begin{scope}[decoration={
								markings,
								mark=at position 0.6 with {\arrow{triangle 60}}}
								] 
						\coordinate (A) at (180+0:2);
						\coordinate (B) at (180+72:2);
						\coordinate (C) at (180+144:2);
						\coordinate (D) at (180+216:2);
						\coordinate (E) at (180+288:2);
						\fill (A) circle (.1);
						\fill (B) circle (.1);
						\fill (C) circle (.1);
						\fill (D) circle (.1);
						\fill (E) circle (.1);

						\draw[postaction={decorate}] (A) -- (B);
						\draw[postaction={decorate}] (B) -- (C);
						\draw[postaction={decorate}] (A) -- (E);
						\draw[postaction={decorate}] (E) -- (D);
						\draw[postaction={decorate}] (D) -- (C);

						\draw[thick,green] (A) circle (.2);
						\draw[thick,red] (C) circle (.2);

					\end{scope}
				\end{tikzpicture}
				\caption{The oriented exchange graph of type $A_2$.}\label{fig:pentagon}
			\end{center}
		\end{figure}

		\begin{figure}[H]
			\begin{center}
				\begin{tikzpicture}[scale = .75]
					\begin{scope}[decoration={
								markings,
								mark=at position 0.6 with {\arrow{triangle 60}}}
								] 
						\coordinate (A) at (180+0:2);
						\coordinate (B) at (180+60:2);
						\coordinate (C) at (180+120:2);
						\coordinate (D) at (180+180:2);
						\coordinate (E) at (180+240:2);
						\coordinate (F) at (180+300:2);
						\fill (A) circle (.1);
						\fill (B) circle (.1);
						\fill (C) circle (.1);
						\fill (D) circle (.1);
						\fill (E) circle (.1);
						\fill (F) circle (.1);

						\draw[postaction={decorate}] (A) -- (B);
						\draw[postaction={decorate}] (B) -- (C);
						\draw[postaction={decorate}] (A) -- (F);
						\draw[postaction={decorate}] (F) -- (E);
						\draw[postaction={decorate}] (E) -- (D);
						\draw[postaction={decorate}] (D) -- (C);

						\draw[thick,green] (A) circle (.2);
						\draw[thick,red] (C) circle (.2);

					\end{scope}
				\end{tikzpicture}
				\caption{The oriented exchange graph of type $B_2$ or $C_2$.}
			\end{center}
		\end{figure}
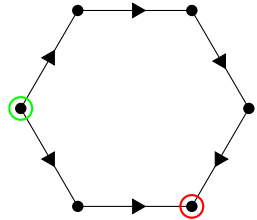

		\begin{figure}[H]
			\begin{center}
				\begin{tikzpicture}[scale = .75]
					\begin{scope}[decoration={
								markings,
								mark=at position 0.6 with {\arrow{triangle 60}}}
								] 
						\coordinate (A) at (180+0:2);
						\coordinate (B) at (180+45:2);
						\coordinate (C) at (180+90:2);
						\coordinate (D) at (180+135:2);
						\coordinate (E) at (180+180:2);
						\coordinate (F) at (180+225:2);
						\coordinate (G) at (180+270:2);
						\coordinate (H) at (180+315:2);
						\fill (A) circle (.1);
						\fill (B) circle (.1);
						\fill (C) circle (.1);
						\fill (D) circle (.1);
						\fill (E) circle (.1);
						\fill (F) circle (.1);
						\fill (G) circle (.1);
						\fill (H) circle (.1);

						\draw[postaction={decorate}] (A) -- (B);
						\draw[postaction={decorate}] (B) -- (C);
						\draw[postaction={decorate}] (A) -- (H);
						\draw[postaction={decorate}] (H) -- (G);
						\draw[postaction={decorate}] (G) -- (F);
						\draw[postaction={decorate}] (F) -- (E);
						\draw[postaction={decorate}] (E) -- (D);
						\draw[postaction={decorate}] (D) -- (C);

						\draw[thick,green] (A) circle (.2);
						\draw[thick,red] (C) circle (.2);

					\end{scope}
				\end{tikzpicture}
				\caption{The oriented exchange graph of type $G_2$.}
			\end{center}
		\end{figure}
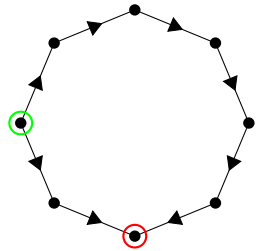

		\begin{figure}[H]
			\begin{center}
				\begin{tikzpicture}[scale = .75]
					\begin{scope}[decoration={
								markings,
								mark=at position 0.6 with {\arrow{triangle 60}}}
								] 
						\coordinate (A) at (0,0);
						\coordinate (B) at (2,0);
						\coordinate (C) at (4,0);
						\coordinate (D) at (-2,0);
						\coordinate (E) at (-4,0);
						\coordinate (H) at (6,0);
						\coordinate (I) at (8,0);

						\fill (A) circle (.1);
						\fill (B) circle (.1);
						\fill (C) circle (.1);
						\fill (D) circle (.1);
						\fill (H) circle (.1);

						\draw[postaction={decorate}] (A) -- (B);
						\draw[postaction={decorate}] (B) -- (C);
						\draw[postaction={decorate}] (A) -- (D);
						\draw[dashed,postaction={decorate}] (D) -- (E);
						
						\draw[postaction={decorate}] (H) -- (C);
						\draw[dashed,postaction={decorate}] (I) -- (H);

						\draw[thick,green] (A) circle (.2);
						\draw[thick,red] (C) circle (.2);

					\end{scope}
				\end{tikzpicture}
				\caption{Oriented exchange graph of of rank two in infinite type.}\label{fig:oeg2infinite}
			\end{center}
		\end{figure}
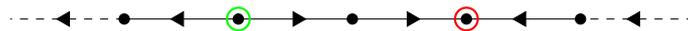

	\subsection{Examples of simply-laced affine types}\label{section:exmpaffines}
		For a quiver $Q$ of affine type, we have (non-maximal) green sequences of infinite lengths. However, Theorem \ref{theorem:affine} asserts that the number of maximal green sequences is finite. There is therefore a maximal length for the maximal green sequences. If Conjecture \ref{conj:intervals} holds, in order to list all the maximal green sequences, it is enough to find some $l \geq 1$ for which $0 < |\green_{\leq l}(Q)| = |\green_{\leq l+1}(Q)|$ and then $\green(Q) = \green_{\leq l}(Q)$ and $\ell^0_{\max}(Q) = \ell_{\max}(Q)$. Using the \verb\Quiver Mutation Explorer\, one can compute several such empirical lengths. 
		
		For $1 \leq n \leq 7$, if $Q$ is a quiver of type $\widetilde A_{n,1}$, that is a quiver of type $\widetilde A$ with $n$ arrows going clockwise and one arrow going counterclockwise, we computed that for any $n \leq 7$, we have 
		$$\ell^0_{\max}(Q) = \frac{n(n+3)}{2}.$$

		For $4 \leq n \leq 7$, if $Q_n$ denotes the following quiver of type $\widetilde D_n$
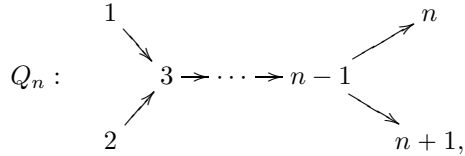
		$$\xymatrix@=1em{
				& 1 \ar[rd] &&&& n \\
			Q_n: 	& & 3 \ar[r] & \cdots \ar[r] & n-1 \ar[ru] \ar[rd] \\
				& 2 \ar[ru] &&&& n+1, \\
		}$$
		we obtained 
		$$\ell^0_{\max}(Q_{n}) = 2n^2 - 2n - 2.$$
		
		Also, it is interesting to note that for (acyclic) quivers of type $\widetilde D_4$, we could compute that the number of maximal green sequences depends on the choice of the orientation of the quiver but the minimal length and the empirical maximal length do not. We do not know if this is a general phenomenon.
%
%

	\subsection{An example from a surface without boundary}\label{ssection:sphere}
		Consider the following triangulation $T$ of the sphere with four punctures. 
		
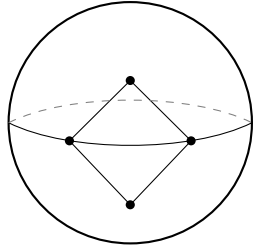
\begin{center}
			\begin{tikzpicture}[scale = .4]
				\draw[thick] circle (4);

				\coordinate (W) at (-4,0);
				\coordinate (E) at (4,0);

				\coordinate (P1) at (-2,-.6);
				\coordinate (P2) at (2,-.6);
				\coordinate (P3) at (0,1.4);
				\coordinate (P4) at (0,-2.72);

				\fill (P1) circle (.15);
				\fill (P2) circle (.15);
				\fill (P3) circle (.15);
				\fill (P4) circle (.15);

				\draw (P1) -- (P3) -- (P2) -- (P4) -- cycle;

				\draw[] (W) .. controls (-2,-1) and (2,-1) .. (E);
				\draw[gray,dashed] (W) .. controls (-2,1) and (2,1) .. (E);
			\end{tikzpicture}
		\end{center}
		As defined in \cite{FST:surfaces}, the quiver $Q_T$ corresponding to this triangulation is the following.
		
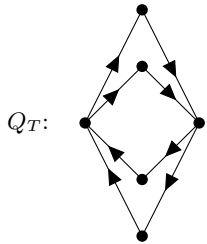
\begin{center}
			\begin{tikzpicture}[scale = .75]
				\tikzstyle{every node} = [font = \small]
				\begin{scope}[decoration={
					markings,
					mark=at position 0.6 with {\arrow{triangle 45}}}
					] 

					\coordinate (1) at (-1,0);
					\coordinate (2) at (1,0);
					\coordinate (3) at (0,1);
					\coordinate (4) at (0,-1);
					\coordinate (5) at (0,2);
					\coordinate (6) at (0,-2);

					\fill (1) circle (.1);
					\fill (2) circle (.1);
					\fill (3) circle (.1);
					\fill (4) circle (.1);
					\fill (5) circle (.1);
					\fill (6) circle (.1);

					\draw[postaction={decorate}] (1) -- (3);
					\draw[postaction={decorate}] (1) -- (5);
					\draw[postaction={decorate}] (3) -- (2);
					\draw[postaction={decorate}] (5) -- (2);
					\draw[postaction={decorate}] (2) -- (4);
					\draw[postaction={decorate}] (2) -- (6);
					\draw[postaction={decorate}] (4) -- (1);
					\draw[postaction={decorate}] (6) -- (1);

					\fill (-2,0) node {$Q_T$:};
			
				\end{scope}
			\end{tikzpicture}
		\end{center}
		Then a direct computation shows that 
		$$\ell_{\min}(Q_T) = 12, \, \ell^0_{\max}(Q_T) = 46$$
		and 
		$$|\green^0(Q_T)| = 1~044~863~666~576.$$

	\subsection{An exceptional mutation-finite type}
		Consider the following quiver of type $\mathbb X_6$
		$$\xymatrix@=1em{
				&	& 2 \ar@<-2pt>[r]\ar@<+2pt>[r] & 3 \ar[ld] \\
			Q:	& 6 	& 1 \ar[l] \ar[u] \ar[d] \\
				&	& 4 \ar@<-2pt>[r]\ar@<+2pt>[r] & 5 \ar[lu]
		}
		$$
		which first appeared in \cite{DerksenOwen:finite} as an example of mutation-finite quiver which is not arising from a surface. Then we obtained
		$$\ell_{\min}(Q) = 10, \, \ell^0_{\max}(Q) =30\, \text{ and } |\green^0(Q)| = 119~819~022. $$

		We could not find a maximal green sequence for the type $\mathbb X_7$ appearing in \cite{DerksenOwen:finite}.


\end{document}